\documentclass{article}
\usepackage{amssymb}
\usepackage{amsmath}
\usepackage{amsthm}
\usepackage{verbatim}
\usepackage{color}
\usepackage{url}

\newtheorem{thm}{Theorem}

\newtheorem{lem}{Lemma}

\newcommand{\sabs}[1]{\left|#1\right|}
\newcommand{\sparen}[1]{\left(#1\right)}
\newcommand{\norm}[1]{\sabs{\sabs{#1}}}

\numberwithin{equation}{section}

\title{A Parametrix Construction for the Wave Equation with Low Regularity Coefficients Using a Frame of Gaussians}
\author{Alden Waters}
\date{February 3, 2010}

\begin{document}
\maketitle

\begin{abstract}
We construct a frame of complex Gaussians for the space of $L^2(\mathbb{R}^n)$ functions.  When propagated along bicharacteristics for the wave equation, the frame can be used to build a parametrix with suitable error terms.  When the coefficients of the wave equation have more regularity, propagated frame functions become Gaussian beams. 
\end{abstract}

\section*{Introduction}

In \emph{A Parametrix Construction for the Wave Equation with $C^{1,1}$ Coefficients} Hart Smith constructed a parametrix solution for the wave equation using a frame that is now called curvelets.  In this paper we construct a new frame out of Gaussian functions.  When a Gaussian function is  propagated along the ray, it becomes a Gaussian beam which looks like a Gaussian distribution on planes perpendicular to a ray in space-time. While the existence of such solutions has been known to the pure mathematics community since the 1960's, there has been a recent revival of interest given their robustness in approximating solutions to PDE's.

Recently Nicolay Tanushev numerically simulated mountain waves with a high degree of accuracy using superpositions of high frequency Gaussian beams. Because Gaussian beams are concentrated along a single ray, it is desirable to use many of them to represent a solution because a global solution is rarely concentrated along a single curve. Tanushev's thesis showed that Gaussian beams have several major advantages over other techniques used to numerically approximate the solution to a mountain wave. Motivated by these numerical calculations we will show that a frame consisting entirely of complex Gaussians can be used to build an accurate parametrix to the wave equation. The idea of using complex Gaussians to build a parametrix is not new. In his paper \emph{Strichartz Estimates for Operators with Nonsmooth Coefficients and the Nonlinear Wave Equation}, Daniel Tataru constructed a parametrix to the wave equation with low regularity coefficients using a modified FBI transform. While the solution in his paper is elegant, numerical calculations with such a construction would be difficult if not impossible. Representing initial data in terms of a frame of Gaussians may lead to more viable and accurate numerical solutions.  

For the rest of this paper we will consider the wave equation 
\[
\partial_t^2u(t,x)-A(t,x,\partial_x)u(t,x)=\partial_t^2u(t,x)-\sum\limits_{1\leq i,j\leq n}a_{ij}(t,x)\partial_{x_i}\partial_{x_j}u(t,x)=0
\]
and we let 
\[
A(x,t)=\{a_{ij}(x,t)\}_{1\leq i,j\leq n}.
\]
 We assume that the matrix $A$ is uniformly positive definite and bounded, that is there exists a constant $C>0$ with  
\[
\frac{|\xi|^2}{C}\leq \sum\limits_{1\leq i,j\leq n}a_{ij}(t,x)\xi_i\xi_j \leq C|\xi|^2
\]
for all $(t,x,\xi)$ in $[-T,T]\times \mathbb{R}^n\times \mathbb{R}^n$. Here $T$ is fixed and finite. Furthermore we assume the entries of the matrix, $a_{ij}(x,t)$ with $(t,x)\in [-T,T]\times \mathbb{R}^n$, are in $C^{1,1}$. $C^{1,1}$ coefficients are of interest because they are minimally  regular. They satisfy a Lipschitz condition in $x$ and $t$ 
\[
|a_{ij}(t,x)-a_{ij}(t',x')|\leq C(|t-t'|+|x-x'|)
\]
and their  first derivatives in $x$ satisfy a Lipschitz condition 
\[
|\nabla_xa_{ij}(t,x)-\nabla_xa_{ij}(t,x')|\leq C|x-x'|
\]

This paper is divided into three major parts. The focus of section 1 is the introduction of a frame of Gaussian functions,  which will represent elements of the Hilbert space $L^2(\mathbb{R}^n)$. Section 2 introduces Gaussian beams and contains the necessary estimates for the construction of a parametrix for the wave equation with $C^{1,1}$ coefficients. Finally section 3 contains the actual parametrix construction. 
\section{Construction of the Frame}

Let the set of functions $\{\phi_{\gamma}(x)\}_{\gamma\in\Gamma}$ be defined as follows:
\[
\phi_{\gamma}(x)=\sparen{\frac{|\xi_{\gamma}|\Delta x_{\gamma}}{2\pi}}^{\frac{n}{2}}\exp \sparen{i\xi_{\gamma}\cdot(x-x_{\gamma})-|\xi_{\gamma}||x-x_{\gamma}|^2}
\]
where $\gamma$ is the index $\gamma=(i,k,\alpha)$, $i\in I_k$, where $I_k$ is a finite subset of integers that depends  on $k$, $k\in \mathbb{N}$, and $\alpha\in\mathbb{Z}^n$. In the first two lemmas we will pick $\xi_{\gamma}=2^k\omega_{i,k}$ a vector in $\mathbb{R}^n$ with $\frac{1}{2}\leq |\omega_{i,k}|<1$ and $x_{\gamma}=\Delta x_{\gamma}\alpha$ another vector in $\mathbb{R}^n$ with $\Delta x_{\gamma}$ a scale factor depending on $k$. We will show that these vectors can be chosen so that the set of functions $\{\phi_{\gamma}(x)\}_{\gamma\in\Gamma}$ form a frame for $L^2(\mathbb{R}^n)$. Not only will our chosen set of $\{\phi_{\gamma}(x)\}_{\gamma\in\Gamma}$ form a frame in $L^2(\mathbb{R}^n)$, but we will show that weighted sequences of frame functions are comparable to the $m^{th}$ Sobolev norm (provided it exists) of any $f(x)$. In particular:

\begin{thm}\label{theoremone}
For any finite $m\geq 0$ and $f(x)\in H^{m}(\mathbb{R}^n)$ there exist constants $C_1$ and $C_2$ independent of $\gamma$ with $0<C_1\leq C_2$ such that the following holds:
\begin{equation}\label{framecondition}
0<C_1\norm{f(x)}_ {\dot{H}^{m}(\mathbb{R}^n)}^2\leq \sum\limits_{\gamma}|2^{km}c(\gamma)|^2\leq C_2 \norm{f(x)}_ {\dot{H}^{m}(\mathbb{R}^n)}^2
\end{equation}
with 
 \[
c(\gamma)=\int_{\mathbb{R}^n}\overline{\phi_{\gamma}(x)}f(x)\,dx.
\]
\end{thm}
For this paper we will use the convention that the Fourier Transform for a function $h(u)\in L^2(\mathbb{R}^n)$ is defined as
\[
\widehat{h(\eta)}=\int\limits_{\mathbb{R}^n} e^{-i\eta\cdot u}h(u)\,du.
\]
We will also need to introduce the following functions 
\[
\psi_{\gamma}(w)=\sparen{\frac{|\xi_{\gamma}|}{2\pi}}^{\frac{n}{2}}\exp \sparen{i\xi_{\gamma}\cdot w-|\xi_{\gamma}||w|^2}.
\]
The only difference between $\psi_{\gamma}(x-u)$ and $\phi_{\gamma}(x)$ is that the discrete variable $x_{\gamma}$ is now a continuous one, $u$, and there is no factor of $(\Delta x_{\gamma})^{\frac{n}{2}}$. Here we note that 
\[
|\widehat{\psi_{\gamma}(\xi)}|^2=2^{-n}\exp\sparen{-\frac{|\xi-\xi_{\gamma}|^2}{2|\xi_{\gamma}|}}.
\]
In Lemma \ref{lemmaone} we construct an approximate partition of unity from the sum of the squares of the Fourier transforms of the $\psi_{\gamma}(w)$. 

\begin{lem}\label{lemmaone}
One can chose $\omega_{i,k}$, $i\in I_k, k\in \mathbb{N}$ with $\frac{1}{2}\leq |\omega_{i,k}|<1$ so that the inequalities 
\begin{equation}\label{partition} 
0<C'_1|\xi|^{2m}\leq \sum\limits_{(i,k)}2^{2km}\exp\sparen{-\frac{|\xi-2^k\omega_{i,k}|^2}{2|2^k\omega_{i,k}|}}\leq C'_2|\xi|^{2m}
\end{equation}
hold for all $\xi\in \mathbb{R}^n$, $m\geq 0$, finite. Here $C'_1$ and $C'_2$ are constants independent of $\xi$.
\end{lem}
For clarity, we will save the proof of Lemma \ref{lemmaone} and Lemma \ref{lemmatwo} below until the end of the proof of  Theorem \ref{theoremone}. For Lemma \ref{lemmatwo}, we will pick $\Delta x_{\gamma}$ so that we can approximate the center term in the inequality (\ref{framecondition}) by an expression which no longer involves $\alpha$, effectively turning the summation over $\alpha$ into an integral. 
\begin{lem}\label{lemmatwo}
For fixed $k\in\mathbb{N}$, let $\Delta x_{\gamma}$ be equal $C_{\epsilon}2^{-\frac{k}{2}-\epsilon k}$ with $\epsilon>0$ and $C_{\epsilon}$ a small constant independent of k and dependent on $\epsilon$. Then for every $\epsilon>0$ there exists a choice of $C_{\epsilon}$ such that the following holds
\begin{align*}\label{norminequality}
&\left|
\sum\limits_{\gamma}\int\limits_{\mathbb{R}^n}\int\limits_{\mathbb{R}^n}2^{2km}f(x)\phi_{\gamma}(x-\alpha\Delta x_{\gamma})\overline{f(x')}\overline{\phi_{\gamma}(x'-\alpha\Delta x_{\gamma})} \,dx\,dx'\right. -  \\ \nonumber&
\left. \sum\limits_{(i,k)}\int\limits_{\mathbb{R}^n}\int\limits_{\mathbb{R}^n}\int\limits_{\mathbb{R}^n}2^{2km}\psi_{\gamma}(x-u)f(x)\overline{\psi_{\gamma}(x'-u)f(x')}\,du\,dx\,dx' \right| \leq \frac{\pi^ne^{-1}}{2}\norm{f}_{L^{2}(\mathbb{R}^n)}^2
\end{align*} 
\end{lem}  
\begin{proof}[ Proof of Theorem \ref{theoremone}]
If we let
\[
\int\limits_{\mathbb{R}^n}\int\limits_{\mathbb{R}^n}\int\limits_{\mathbb{R}^n}2^{2km}\psi_{\gamma}(x-u)\overline{\psi_{\gamma}(x'-u)}f(x)f(x')\,du\,dx\,dx'
\]
then the kernel of this expression can be rewritten as
\[
\int\limits_{\mathbb{R}^n}2^{2km}\psi_{\gamma}(x-u)\overline{\psi_{\gamma}(x'-u)}\,du=(2\pi)^n\int\limits_{\mathbb{R}^n}e^{i\xi\cdot(x-x')}2^{2km}|\widehat{\psi_{\gamma}(\xi)}|^2\,d\xi
\]
since the Fourier Transform is an isometry on $L^2(\mathbb{R}^n)$. As remarked earlier, 
\[
|\widehat{\psi_{\gamma}(\xi)}|^2=2^{-n}\exp\sparen{-\frac{|\xi-\xi_{\gamma}|^2}{2|\xi_{\gamma}|}}
\]
so that by Lemma \ref{lemmaone} and Fubini's Theorem
\begin{align*}
&\pi^nC'_1\norm{|\xi|^{m}\widehat{f(\xi)}}_{L^2(\mathbb{R}^n)}^2
\leq \sum\limits_{(i,k)}\int\limits_{\mathbb{R}^n}\int\limits_{\mathbb{R}^n}\int\limits_{\mathbb{R}^n}2^{2km}\psi_{\gamma}(x-u)\overline{\psi_{\gamma}(x'-u)}f(x)f(x')\,du\,dx\,dx'  \\ &\leq \nonumber \pi^nC'_2\norm{|\xi|^{m}\widehat{f(\xi)}}_{L^2(\mathbb{R}^n)}^2
\end{align*}
which is equivalent to 
\begin{align}\label{inequalityAf}
&\pi^nC'_1\norm{f(x)}_{\dot{H}^{m}(\mathbb{R}^n)}^2\leq \sum\limits_{(i,k)} \int\limits_{\mathbb{R}^n}\int\limits_{\mathbb{R}^n}\int\limits_{\mathbb{R}^n}2^{2km}\psi_{\gamma}(x-u)\overline{\psi_{\gamma}(x'-u)}f(x)f(x')\,du\,dx\,dx'
 \\ &\leq \nonumber \pi^nC'_2\norm{f(x)}_{\dot{H}^{m}(\mathbb{R}^n)}^2.
\end{align}
From Lemma \ref{lemmatwo} we also have
\begin{align}\label{approx}
&\left|
\sum\limits_{\gamma}\int\limits_{\mathbb{R}^n}\int\limits_{\mathbb{R}^n}2^{2km}f(x)\phi_{\gamma}(x-\alpha\Delta x_{\gamma})\overline{f(x')}\overline{\phi_{\gamma}(x'-\alpha\Delta x_{\gamma})} \,dx\,dx'\right. - \\ \nonumber &
\left. \sum\limits_{(i,k)}\int\limits_{\mathbb{R}^n}\int\limits_{\mathbb{R}^n}\int\limits_{\mathbb{R}^n}2^{2km}\psi_{\gamma}(x-u)f(x)\overline{\psi_{\gamma}(x'-u)f(x')}\,du\,dx\,dx' \right|  \leq 
\nonumber \frac{\pi^ne^{-1}}{2}\norm{f}_{L^2(\mathbb{R}^n)}^2,
\end{align}
but since $\frac{e^{-1}}{2}<C'_1$ and $C'_2>C'_1$ we can combine inequalities (\ref{inequalityAf}) and (\ref{approx}) to conclude 
\begin{align*}
&C_1\norm{f}_{\dot{H}^{m}(\mathbb{R}^n)}^2 \leq 
\left|
\sum\limits_{\gamma}\int\limits_{\mathbb{R}^n}\int\limits_{\mathbb{R}^n}2^{2km}f(x)\phi_{\gamma}(x-x_{\gamma})\overline{f(x')}\overline{\phi_{\gamma}(x'-x_{\gamma})} \,dx\,dx' \right| \\ 
& \leq \nonumber C_2\norm{f}_{\dot{H}^{m}(\mathbb{R}^n)}^2
\end{align*}
which is the result (\ref{framecondition}).
\end{proof}

\begin{proof}[Proof of Lemma \ref{lemmaone}]

Since $\xi\in\mathbb{R}^n/\{0\}$, we begin by considering $\mathbb{R}^n$ as an infinite union of dyadic annuli, each of which we will cover with real Gaussians which are centered at our choice of $2^k\omega_{i,k}$. In every annulus $2^{k-1}\leq|\xi|< 2^{k}$, for all $k \in \mathbb{N}$, we choose the vectors $2^k\omega_{i,k}$ such that for all $i\neq j$ $|2^k\omega_{i,k}-2^k\omega_{j,k}|> 2^{\frac{k}{2}}$, but that the number of $2^k\omega_{i,k}$ in each annulus is as large as possible.  The index set $I_k$ is finite as the volume of every annulus is finite. 

Fixing $\xi$ for the rest of this proof, $\xi$ must lie in an annulus $2^{k-1}\leq |\xi|<2^k$ for some fixed $k$ in $\mathbb{N}$. As a result of our choice of vectors, for all $\xi\in \mathbb{R}^n$ there exists at least one point $2^k\omega_{i,k}$ for which the inequality $|\xi-2^k\omega_{i,k}|< 2^{\frac{k}{2}}$ holds. This condition gives the existence of a lower bound: 
\[
|\xi|^{2m}e^{-1}\leq \sum\limits_{(i,k)}2^{2km}\exp\sparen{-\frac{|\xi-2^k\omega_{i,k}|^2}{2|2^k\omega_{i,k}|}}.
\]

To show the sum is bounded above, we will consider sets of indices $A,B,C,D,$ and $E$ whose union contains all the indices $(i,k)$ in $\gamma$ and show that the contribution to the sum from each of these sets is bounded by a constant multiple of $|\xi|^{2m}$.  The cases $k=0,1$ are easy, so we consider $k\geq 2$ 

First let $\mathcal{A}$ consist of those indices $(i,q)$ for which $|\xi-2^q\omega_{i,q}|<2^{\frac{k}{2}}$. Clearly $q$ can only be equal to $k-1,k,$ or $k+1$. Fixing $q$ for the moment and setting $r=2^{\frac{q}{2}}$, if we consider a ball $B$ of radius $\frac{r}{2}$ centered at each $2^q\omega_{i,q}$ then for all pairs $(i,q), (j,q)\in \mathcal{A}$ with $i\neq j$, $B(2^q\omega_{i,q},\frac{r}{2})\cap B(2^q\omega_{j,q},\frac{r}{2})=\emptyset$. But by the triangle inequality, all balls of radius $\frac{r}{2}$ with centers that have indices in $\mathcal{A}$ are contained in a ball of radius $\frac{3r}{2}$ around $\xi$. Therefore the total number of balls $N$ is bounded as
\begin{align*}
Vol\sparen{B\sparen{\xi,\frac{3r}{2}}}\geq N Vol\sparen{B\sparen{2^q\omega_{i,q},\frac{r}{2}}}
\end{align*}
which implies $N\leq 3^n$. Since there are only three possible values $q$ can take, the total contribution for the set $\mathcal{A}$ to the sum is bounded by $3^{n+1}.$

For the second set let $\mathcal{B}$ consist of those indices $(i,q)$ for which the inequality $2^{\frac{k}{2}}\leq |\xi-2^{q}\omega_{i,q}|< 2^{k}$ holds and also $|k-q|\leq 1$. We can write $\mathcal{B}$ as a collection of subsets $\mathcal{B}_j$ such that 
\[
\mathcal{B}=\bigcup\limits_{j=2}^{2^{\frac{k}{2}}}\mathcal{B}_j
\]
where $\mathcal{B}_j$ denotes the set of indices for which $(j-1)2^{\frac{k}{2}}\leq|\xi-2^{q}\omega_{i,q}|< j2^{\frac{k}{2}}$. As before, we consider balls of radius $\frac{r}{2}=2^{\frac{q}{2}-1}$ centered at each $2^q\omega_{i,q}$ such that for all pairs $(i,q), (j,q)\in \mathcal{B}_j$ with $i\neq j$, $B(2^q\omega_{i,q},\frac{r}{2})\cap B(2^q\omega_{j,q},\frac{r}{2})=\emptyset$. By the triangle inequality, all balls with centers that have indices in $\mathcal{B}_j$ are contained in an annulus centered about $\xi$ with inner radius $(j-1)r-\frac{r}{2}$ and outer radius $jr+\frac{r}{2}$. The total number of indices for fixed $q$ in each set $\mathcal{B}_j$ is bounded as 
\begin{align*}
Vol\sparen{B\sparen{\xi,jr+\frac{r}{2}}}-Vol\sparen{B\sparen{\xi,(j-1)r-\frac{r}{2}}}\geq NVol\sparen{B\sparen{2^q\omega_{i,q},\frac{r}{2}}}
\end{align*}
which implies
\begin{align*}
N\leq 2^n\sparen{\sparen{j+\frac{1}{2}}^n-\sparen{j-\frac{3}{2}}^n}.
\end{align*}
Since $q$ can take only three possible values, multiplying this last bound by $3$ gives a bound on the total number of indices in each set $B_j$.  
Because of the restriction on the size of $|\xi-2^{q}\omega_{i,q}|$ and the fact that $|q-k|\leq 1$, the inequality
\[
\frac{(j-1)^2}{2}\leq \frac{|\xi-2^q{\omega_{i,q}}|^2}{2|2^q\omega_{i,q}|}\leq 4j^2
\]
holds for each tuple in $\mathcal{B}_j$.
Summing over all of the sets $\mathcal{B}_j$ 
\begin{align*}
&\sum_j\sum\limits_{(i,q)\in\mathcal{B}_j}2^{2km}\exp\sparen{\frac{-|\xi-2^q{\omega_{i,q}}|^2}{|2^q\omega_{i,q}|}}\\&<  \sum\limits_{j=2}^{2^{\frac{k}{2}}} 2^{2km}3(2^n)\sparen{\sparen{j+\frac{1}{2}}^n-\sparen{j-\frac{3}{2}}^n}\exp\sparen{-\frac{(j-1)^2}{2}} \\ &
\leq \sum\limits_{j=2}^{\infty}2^{2km}3(2^n)(j+1)^n\exp\sparen{-\frac{(j-1)^2}{2}}.
\end{align*}
The sum 
\begin{align*}
\sum\limits_{j=2}^{\infty}3(2^n)(j+1)^n\exp\sparen{-\frac{(j-1)^2}{2}}
\end{align*}
is finite; furthermore it is uniformly bounded regardless of the choice of $k$ and hence of $\xi$. Therefore since $2^{k-1}\leq |\xi|\leq 2^{k}$, the total contribution from the set $\mathcal{B}$ is bounded by a constant times $|\xi|^{2m}$.

Next, let $\mathcal{C}$ be the set of indices $(i,q)$ for which $|\xi-2^{q}\omega_{i,q}|> 2^{k}$ holds and also $|k-q|\leq 1$. As before for each fixed $q$, we take balls of radius $\frac{r}{2}=2^{\frac{q}{2}-1}$ centered at each $2^q\omega_{i,q}$ so that for all pairs $(i,q), (j,q)\in \mathcal{C}$ with $i\neq j$, $B(2^q\omega_{i,q},\frac{r}{2})\cap B(2^q\omega_{j,q},\frac{r}{2})=\emptyset$.  All balls with centers that have indices in $\mathcal{C}$ are contained in an annulus centered about the origin with inner radius $r^2-\frac{r}{2}$ and outer radius $r^2+\frac{r}{2}$. Since we have removed a number of the vectors because their indices are in $\mathcal{B}$, the total number of indices, $N$ for fixed $q$ is over-estimated as follows  
\begin{align*}
Vol\sparen{B\sparen{0,r^2+\frac{r}{2}}}-Vol\sparen{B\sparen{0,r^2-\frac{r}{2}}}\geq NVol\sparen{B\sparen{2^q\omega_{i,q},\frac{r}{2}}}
\end{align*}
which implies
\begin{align*}
N\leq 2^n\sparen{\sparen{r+\frac{1}{2}}^n-\sparen{r-\frac{1}{2}}^n}.
\end{align*}
 Since $|\xi-2^{q}\omega_{i,q}|>2^{k}$ for all $(i,q)\in \mathcal{C}$, the inequality 
\[
2^{k-2}<\frac{2^{2k}}{2^{q+1}}\leq \frac{|\xi-2^{q}\omega_{i,q}|^2}{2|2^q\omega_{i,q}|}
\]
holds for each point $2^q\omega_{i,q}$ with indices in $\mathcal{C}$. 
Then the contribution from the set $\mathcal{C}$ is bounded in terms of a sum over $k$ as
\begin{align*}
&\sum\limits_{(i,q)\in\mathcal{C}}2^{2km}\exp\sparen{\frac{-|\xi-2^q{\omega_{i,q}}|^2}{2|2^q\omega_{i,q}|}}\\ &<\sum_{q=k-1}^{k+1}2^{2km}2^n\sparen{\sparen{2^{\frac{q}{2}}+\frac{1}{2}}^n-\sparen{2^{\frac{q}{2}}-\frac{1}{2}}^n}\exp\sparen{-2^{k-2}}. 
\end{align*} 
But since
\begin{align}\label{sumsetc}
&\sum_{q=k-1}^{k+1}2^n\sparen{\sparen{2^{\frac{q}{2}}+\frac{1}{2}}^n-\sparen{2^{\frac{q}{2}}-\frac{1}{2}}^n}\exp\sparen{-2^{k-2}}\\ \nonumber &< 3(2^n)\sparen{\sparen{2^{\frac{k+1}{2}}+\frac{1}{2}}^n-\sparen{2^{\frac{k+1}{2}}-\frac{1}{2}}^n}\exp\sparen{-2^{k-2}},
\end{align} 
and $2^{kn}\exp\sparen{-2^{k-2}}\rightarrow 0$ for all finite $n\in \mathbb{R}$ as $k\rightarrow \infty$, (\ref{sumsetc}) is bounded is independent of $\xi$. So we can conclude that  since $2^{k-1}\leq |\xi|\leq 2^{k}$, the total contribution from the set $\mathcal{C}$ is bounded by a constant times $|\xi|^{2m}$ as well.
 
Now let $\mathcal{D}$ be the set of indices $(i,q)$ for which $q<k-1$. To find the number of vectors in $\mathcal{D}$ for fixed $q$ we again take balls of radius $\frac{r}{2}=2^{\frac{q}{2}-1}$ centered at each $2^q\omega_{i,q}$ so that for all pairs $(i,q), (j,q)\in \mathcal{D}$ with $i\neq j$, $B(2^q\omega_{i,q},\frac{r}{2})\cap B(2^q\omega_{j,q},\frac{r}{2})=\emptyset$. By the triangle inequality, all balls with centers that have indices in $\mathcal{D}$ are contained in an annulus centered about the origin with inner radius $r^2-\frac{r}{2}$ and outer radius $r^2+\frac{r}{2}$. The total number of indices $N$ is bounded as 
\begin{align*}
Vol\sparen{B\sparen{0,r^2+\frac{r}{2}}}-Vol\sparen{B\sparen{0,r^2-\frac{r}{2}}}\geq NVol\sparen{B\sparen{2^q\omega_{i,q},\frac{r}{2}}}
\end{align*}
which gives
\begin{align*}
N\leq 2^n\sparen{\sparen{r+\frac{1}{2}}^n-\sparen{r-\frac{1}{2}}^n}.
\end{align*}
We can conclude there are at most $2^{n}\sparen{\sparen{2^{\frac{q}{2}}+\frac{1}{2}}^n-\sparen{2^{\frac{q}{2}}-\frac{1}{2}}^n}$ vectors for fixed $q$.  Since for these indices $q<k-1$, the inequality
\[
2^{q-1}\leq  \frac{(2^{k-1}-2^{q})^2}{2^{q+1}}\leq \frac{|\xi-2^{q}\omega_{i,q}|^2}{2|2^q\omega_{i,q}|}
\]
holds for each $(i,q)$ in $\mathcal{D}$.
 The total contribution from the set $\mathcal{D}$ is also bounded by a constant times $|\xi|^{2m}$
\begin{align*}
&\sum\limits_{(i,q)\in\mathcal{D}}2^{2km}\exp\sparen{\frac{-|\xi-2^q{\omega_{i,q}}|^2}{2|2^q\omega_{i,q}|}}\\& <\sum_{q=1}^{k-2}2^{2km}2^n\sparen{\sparen{2^{\frac{q}{2}}+\frac{1}{2}}^n-\sparen{2^{\frac{q}{2}}-\frac{1}{2}}^n}\exp\sparen{-2^{q-1}}\\
&<\sum_{q=1}^{\infty}2^{2km}2^n(2^{\frac{q}{2}}+1)^n\exp\sparen{-2^{q-1}}
\end{align*} 
since the sum 
\[
\sum_{q=1}^{\infty}2^n(2^{\frac{q}{2}}+1)^n\exp\sparen{-2^{q-1}}
\]
is uniformly bounded with respect to $k$ . 

The final set $\mathcal{E}$ contributing to the sum consists of the indices $(i,q)$ for which $q>k+1$. Again, as above the total number of vectors $N$ for fixed $q$ is at most 
\[
2^{n}\sparen{\sparen{2^{\frac{q}{2}}+\frac{1}{2}}^n-\sparen{2^{\frac{q}{2}}-\frac{1}{2}}^n}.
\] 
To find the exponential contribution for each $q>k+1$ note that for each $(i,q)\in\mathcal{E}$
\begin{equation*}
2^{q-5}\leq \frac{(2^{q-1}-2^{q-2})^2}{2^{q+1}} \leq \frac{|\xi-2^{q}\omega_{i,q}|^2}{2|2^{q}\omega_{i,q}|}.
\end{equation*}
Therefore
\begin{align*}
\sum\limits_{(i,q)\in\mathcal{E}}2^{2km}\exp\sparen{\frac{-|\xi-2^q{\omega_{i,q}}|^2}{2|2^q\omega_{i,q}|}}<\sum_{q=k+2}^{\infty}2^{2km}2^n\sparen{2^{\frac{q}{2}}+1}^n\exp\sparen{-2^{q-5}}.
\end{align*}
The sum 
\begin{align*}
\sum_{q=k+2}^{\infty}2^n\sparen{2^{\frac{q}{2}}+1}^n\exp\sparen{-2^{q-5}}
\end{align*}
is convergent and bounded independent of $k$ and $q$. Therefore the total contribution from the set $\mathcal{E}$ is bounded by a constant times $|\xi|^{2m}$ as well. This completes the proof of the Lemma. The construction of the approximate partition of unity is similar in idea to the construction of almost orthogonal frames in Meyer's book \emph{Wavelets}. The $\xi_{\gamma}$ which are further away from the variable $\xi$ contribute less to the the partition than those which are close. 

\end{proof}

\begin{proof}[Proof of Lemma \ref{lemmatwo}]

For convenience we let: 
\[
g_{\gamma}(u,x,x')(\Delta_{\gamma}x)^{n}=2^{2km}\phi_{\gamma}(x-u)\overline{\phi_{\gamma}(x'-u)}
\]
which implies that  the operator   
\[
\sum\limits_{\gamma}\sparen{2^{2km}\phi_{\gamma}(x-\alpha\Delta x_{\gamma})\overline{\phi_{\gamma}(x'-\alpha\Delta x_{\gamma})}}
\]
is equal to 
\[
\sum_{(i,k)}\sum\limits_{\alpha\in\mathbb{Z}^n}(g_{\gamma}(\alpha\Delta x_{\gamma},x,x'))(\Delta x_{\gamma})^n.
\]
We will rewrite the sum over $\alpha$ above using the Poisson summation formula. Recall 
\begin{thm}
Poisson Summation Formula: 
Let $a$ be constant, $h(u)\in \mathcal{S}(\mathbb{R}^n)$, and $\alpha,\beta \in \mathbb{Z}^n$ Then the following holds
\[
a^n\sum\limits_{\alpha\in\mathbb{Z}^n}h(a\alpha)=\sum\limits_{\beta\in\mathbb{Z}^n}\hat{h}\sparen{\frac{2\pi\beta}{a}}
\]
\end{thm}
Since by definition 
\[
\hat{g}_{\gamma}(\eta,x,x')=\int\limits_{\mathbb{R}^n} e^{-i\eta\cdot u}g_{\gamma}(u,x,x')\,du,
\]
so we start by computing
\begin{align*}
&g_{\gamma}(u,x,x')\\&=
2^{2km}\sparen{\frac{|\xi_{\gamma}|}{2\pi}}^{n}
\exp\sparen{i(u-x)\cdot\xi_{\gamma}-|\xi_{\gamma}|(u-x)^2-i(u-x')\cdot\xi_{\gamma}-|\xi_{\gamma}|(u-x')^2}\\&=
2^{2km}\sparen{\frac{|\xi_{\gamma}|}{2\pi}}^{n} \exp\sparen{i(x'-x)\cdot\xi_{\gamma}} \exp\sparen{|\xi_{\gamma}|\sparen{-2u^2+2u(x+x')-x^2-x'^2}}\\&=
2^{2km}\sparen{\frac{|\xi_{\gamma}|}{2\pi}}^{n}\exp\sparen{i(x-x')\cdot\xi_{\gamma}} \exp\sparen{-2|\xi_{\gamma}|\sparen{u-\sparen{\frac{x+x'}{2}}}^2}\exp\sparen{|\xi_{\gamma}|\sparen{\frac{-x^2}{2}-\frac{-x'^2}{2}+xx'}}\\ &=
2^{2km}\sparen{\frac{|\xi_{\gamma}|}{2\pi}}^{n}\exp\sparen{i(x-x')\cdot\xi_{\gamma}} \exp\sparen{-2|\xi_{\gamma}|\sparen{u-\sparen{\frac{x+x'}{2}}}^2}\exp\sparen{-\frac{|\xi_{\gamma}|(x-x')^2}{2}}
\end{align*}
which by a standard result on the Fourier transform of a Gaussian (see Appendix 1) gives 
\begin{align}\label{ftransform}
&\hat{g}_{\gamma}(\eta,x,x')\\ &= \nonumber
2^{2km}\sparen{\frac{\pi}{2|\xi_{\gamma}|}}^{\frac{n}{2}}\sparen{\frac{|\xi_{\gamma}|}{2\pi}}^{n}\exp\sparen{i(x-x')\cdot\xi_{\gamma}+i\eta\cdot\sparen{\frac{x+x'}{2}}}\exp\sparen{-\frac{\eta^2}{8|\xi_{\gamma}|}-\frac{|\xi_{\gamma}|(x-x')^2}{2}}.
\end{align}
Now we notice that
\[
\hat{g}_{\gamma}(0,x,x')=\int\limits_{\mathbb{R}^n}g_{\gamma}(u,x,x')\,du
\]
so in applying Poisson summation formula we obtain
\[
\sum\limits_{\alpha\in\mathbb{Z}^n}(g_{\gamma}(\alpha\Delta x_{\gamma},x,x'))(\Delta x_{\gamma})^n=\int\limits_{\mathbb{R}^n} g_{\gamma}(u,x,x')\,du+\sum\limits_{\beta\in\mathbb{Z}^n \\ \beta\neq 0}\hat{g}_{\gamma}\sparen{\frac{2\pi\beta}{\Delta x_{\gamma}},x,x'}
\]
where $\hat{g}_{\gamma}(\eta,x,x')$ is given explicitly by (\ref{ftransform}). From this we can conclude the left hand side of (\ref{approx}) is
\begin{align*}
\left| \int\limits_{\mathbb{R}^n}\int\limits_{\mathbb{R}^n}\sum\limits_{(k,i)}\sum\limits_{\beta\in\mathbb{Z}^n, \beta\neq 0}\hat{g}_{\gamma}\sparen{\frac{2\pi\beta}{\Delta x_{\gamma}},x,x'}f(x)\overline{f(x')}\,dx \,dx'\right|.
\end{align*}
By symmetry of the integrands in $x$ and $x'$, if we use Schur's lemma then the inequality in (\ref{approx}) follows from the estimate:
\begin{equation}\label{cond}
\sup_{x\in\mathbb{R}^n}\int\limits_{\mathbb{R}^n}\sabs{\sum\limits_{(i,k)}\sum\limits_{\beta\in\mathbb{Z}^n, \beta\neq 0}\hat{g}_{\gamma}\sparen{\frac{2\pi\beta}{\Delta x_{\gamma}},x,x'}}\,dx'<\sqrt{\frac{\pi^ne^{-1}}{2}}.
\end{equation}
 If we examine the argument of the integral in (\ref{cond}) we find from equality (\ref{ftransform})
\begin{align*}
&\sabs{\sum\limits_{\beta\in\mathbb{Z}^n, \beta\neq 0}\hat{g}_{\gamma}\sparen{\frac{2\pi\beta}{\Delta x_{\gamma}},x,x'}}\leq \sum\limits_{\beta\in\mathbb{Z}^n, \beta\neq 0}\sabs{\hat{g}_{\gamma}\sparen{\frac{2\pi\beta}{\Delta x_{\gamma}},x,x'}}\\&=
\sum\limits_{\beta\in\mathbb{Z}^n, \beta\neq 0}2^{2k\alpha}\sparen{\frac{|\xi_{\gamma}|}{8\pi}}^{\frac{n}{2}} \exp\sparen{-\frac{(2\pi\beta)^2}{8|\xi_{\gamma}|(\Delta x_{\gamma})^2}-\frac{|\xi_{\gamma}|(x-x')^2}{2}}.
\end{align*}
Integrating both sides of the above inequality with respect to $x'$ gives 
\begin{equation*}
\sup_{x\in\mathbb{R}^n}\int\limits_{\mathbb{R}^n}\sabs{\sum\limits_{\beta\in\mathbb{Z}^n, \beta\neq 0}\hat{g}_{\gamma}\sparen{\frac{2\pi\beta}{\Delta x_{\gamma}},x,x'}}\,dx'\leq\sum\limits_{\beta\in\mathbb{Z}^n, \beta\neq 0}2^{2km}2^{-n}\exp\sparen{-\frac{(2\pi\beta)^2}{8|\xi_{\gamma}|(\Delta x_{\gamma})^2}}.
\end{equation*}
Let $\beta=(\beta_1,\beta_2, . . . ,\beta_n)$ then since with this notation $\beta_i\in\mathbb{Z}$ is indexed independent of $\beta_j\in \mathbb{Z}$ for all $i\neq j$ we have 
\begin{align*}
\sum\limits_{\beta\in\mathbb{Z}^n}\sparen{\prod\limits_{i=1}^n
\exp\sparen{-\frac{(2\pi\beta_i)^2}{8|\xi_{\gamma}|(\Delta x_{\gamma})^2}}}=\prod\limits_{i=1}^{n}\sparen{\sum\limits_{\beta_i\in\mathbb{Z}}
\exp\sparen{-\frac{(2\pi\beta_i)^2}{8|\xi_{\gamma}|(\Delta x_{\gamma})^2}}}.
\end{align*}
As $\beta\neq 0$, at least one of the $\beta_i's$ must also be nonzero. Without loss of generality take $\beta_n\neq 0$ then
\begin{align}\label{products}
&\sum\limits_{\beta\in\mathbb{Z}^n, \beta\neq 0}\exp\sparen{-\frac{(2\pi\beta)^2}{8|\xi_{\gamma}|(\Delta x_{\gamma})^2}}\\ &\leq  \nonumber n\sparen{\prod\limits_{i=1}^{n-1}\sparen{\sum\limits_{\beta_i\in\mathbb{Z}}
\exp\sparen{-\frac{(2\pi\beta_i)^2}{8|\xi_{\gamma}|(\Delta x_{\gamma})^2}}}}
\sparen{\sum\limits_{\beta_n\in \mathbb{Z}, \beta_n\neq0}\exp\sparen{-\frac{(2\pi\beta_n)^2}{8|\xi_{\gamma}|(\Delta x_{\gamma})^2}}}.
\end{align}
To put a bound on this last expression, we now need to pick $\Delta x_{\gamma}$. Letting 
\begin{align*}
a(k)=\frac{(2\pi)^2}{8|\xi_{\gamma}|(\Delta x_{\gamma})^2}
\end{align*}
since $\Delta x_{\gamma}$ is of the form $C_{\epsilon}2^{-\frac{k}{2}-\epsilon k}$ we have that
\begin{align*}
a(k)=\frac{\pi^22^{\epsilon k}}{2C_{\epsilon}}. 
\end{align*}
Now if we pick $C_{\epsilon}<4$ then for all $k$, $a(k)>1$ always. Therefore for any such choice of $C_{\epsilon}$ we have 
\[
\sum\limits_{\beta_i\in\mathbb{Z}}\exp\sparen{-\frac{(\beta_i)^2}{8|\xi_{\gamma}|(\Delta x_{\gamma})^2}}<2\sum\limits_{\beta_i\in\mathbb{N}}\exp\sparen{-a(k)\beta_i}=\frac{2}{1-e^{-a(k)}}<\frac{2}{1-e^{-1}}
\]
which ensures the first product in (\ref{products}) is uniformly bounded independent of $k$: 
\[
 \sparen{\prod\limits_{i=1}^{n-1}\sparen{\sum\limits_{\beta_i\in\mathbb{Z}}
\exp\sparen{-\frac{(2\pi\beta_i)^2}{8|\xi_{\gamma}|(\Delta x_{\gamma})^2}}}}<\sparen{\frac{2}{1-e^{-a(k)}}}^{n-1}<\sparen{\frac{2e}{e-1}}^{n-1}
\]
From Lemma 1, the number of the $2^k\omega_{i,k}$ can be over-estimated by by $2^n(2^{\frac{k}{2}}+1)^n$ per fixed $k$, so combining estimates 
\begin{align*}
&\sum\limits_{(i,k)}\sum\limits_{\beta\in\mathbb{Z}^n, \beta\neq 0}2^{2km}2^{-n}\exp\sparen{-\frac{(2\pi\beta)^2}{8|\xi_{\gamma}|(\Delta x_{\gamma})^2}}\\
<&\sum\limits_{k=1}^{\infty}2^{2km}n(2^{\frac{k}{2}}+1)^n\sparen{\frac{2e}{e-1}}^{n-1}\exp\sparen{-a(k)} 
\end{align*}
Since $\epsilon>0$, $\exp(-a(k))$ dominates any power of $2^k$. Thus as long as $C_{\epsilon}$ is chosen sufficiently small we can make this sum less than $\sqrt{\frac{\pi^ne^{-1}}{2}}$ which concludes the proof of inequality (\ref{cond}).  

\end{proof}

\section{Operator Norm Estimates}

From Theorem $1$ in Section $1$, the operator $P^m_1(f(y))=\{c(\gamma)\}_{\gamma\in\Gamma}$ where
 \[
c(\gamma)=\int_{\mathbb{R}^n}2^{km}\overline{\phi_{\gamma}(x)}f(x)\,dx
\]
is a one-to-one bounded mapping of  $H^m(\mathbb{R}^n)$ into the space of sequences weighted with $2^{km}$ which are convergent in $l^2(\Gamma)$. Let $P^m_2=P_1^{m*}$ be defined as follows 
\[
P^m_2: l^2(\Gamma)\rightarrow L^2(\mathbb{R}^n), \qquad P_2(\{c(\gamma)\})=\sum\limits_{\gamma}2^{km}c(\gamma)\phi_{\gamma}(y)
\]
Now recall that $T$ is an operator of order $m$ if $T$ maps $H^{r}(\mathbb{R}^n)\rightarrow H^{r-m}(\mathbb{R}^n)$. In section $1$, we showed that $\Pi^m=P^m_2\circ P_1^m$ is an operator of order $2m$. Let $I$ denote the identity operator, as there exist constants $C_1'$ and $C_2'$ such that $C_1'I\leq \Pi^0 \leq C_2'I$, in $L^{2}(\mathbb{R}^n)$ norm sense, $P_1^0$ is bounded and invertible on its range. The construction of $P_1^0$ and $P^0_2$ allows us to translate the characterization of functions and operators in $H^m(\mathbb{R}^n)$ to the framework of weighted sequences in $l^2(\Gamma)$. Armed with the frame operators, we will show that when the frame functions are propagated along bicharacteristics for the wave equation, their Sobolev norm is preserved. This will help us also show that the the action of the operator $T(x,t,\partial_x,\partial_t)=\partial^2_t-A(x,t,\partial_x)$ on the parametrix  is order $1$. The estimates established in this section will ultimately be useful in building the parametrix in section 3.  

First we recall that  $T(x,t,\partial_x,\partial_t)=\partial^2_t-A(x,t,\partial_x)$ has principal symbol $p(x,t,\xi,\tau)=\tau^2-\sum\limits_{i,j}a_{i,j}(x,t)\xi_i\xi_j$. The bicharacteristics associated to $p$ are 
\begin{align}\label{origbichar}
&\frac{d t}{ds}=p_{\tau},  \qquad \frac{d x_j}{d s}= p_{\xi_j}, \qquad 
\frac{d\xi_j}{d s}=-p_{x_j}, \qquad \frac{d \tau}{d s}=-p_t.
\end{align}
Setting $q=\sparen{\sum\limits_{i,j}a_{i,j}\xi_i\xi_j}^{\frac{1}{2}}$, we find that $p=(\tau-q)(\tau+q)$. There are two choices for null bicharacteristics. Here we assume that $\tau=q$, so the bicharacteristic equations (\ref{origbichar}) become 
\begin{align}\label{bicharacteristics}
&\frac{d t}{ds}=p_{\tau}=2\tau=2q , \qquad \frac{d x}{d t}=\frac{p_{\xi}}{2q}=-q_{\xi},\\& \nonumber
\frac{d\xi}{d t}=\frac{-p_x}{2q}=q_x, \qquad \frac{d \tau}{d t}=1
\end{align}
Define 
\[
(x_{\gamma}(t,t',x_{\gamma},\xi_{\gamma}),\xi_{\gamma}(t,t',x_{\gamma},\xi_{\gamma}))
\]
as the solution to the system (\ref{bicharacteristics}) at time $t$ with initial conditions 
\[
(x_{\gamma}(t',t',x_{\gamma},\xi_{\gamma}),\xi_{\gamma}(t',t',x_{\gamma},\xi_{\gamma}))=(x_{\gamma},\xi_{\gamma}) 
\]
where $(x_{\gamma},\xi_{\gamma})$ are given in Lemmas \ref{lemmaone} and \ref{lemmatwo} of Section 1. We let $\mathcal{U}(t,t')$ denote the the evolution operator associated to this transformation. Often we will abbreviate 
\[
\mathcal{U}(t,0)(x_{\gamma},\xi_{\gamma})=(x_{\gamma}(t,0,x_{\gamma},\xi_{\gamma}),\xi_{\gamma}(t,0,x_{\gamma},\xi_{\gamma}))
\]
as 
\[
(x_{\gamma}(t),\xi_{\gamma}(t))
\]
and
\[
\mathcal{U}(0,t)(x_{\gamma},\xi_{\gamma})=(x_{\gamma}(0,t,x_{\gamma},\xi_{\gamma}),\xi_{\gamma}(0,t,x_{\gamma},\xi_{\gamma}))
\]
as 
\[
(x_{\gamma}(-t),\xi_{\gamma}(-t)).
\]
Let
\begin{align*}
&\phi_{\gamma}(t,x)=\sparen{\frac{|\xi_{\gamma}(t)|\Delta x_{\gamma}}{2\pi}}^{\frac{n}{2}} \exp\sparen{i\xi(t)\cdot(x-x_{\gamma}(t))-|\xi_{\gamma}(t)||x-x_{\gamma}(t)|^2}.
\end{align*}
Then define $E(t)$ to be the propagation operator acting on $f(x)\in L^2(\mathbb{R}^n)$ as follows:
\begin{align*}
\Pi^0 E(t) \Pi^0 f & = P^0_2 B_{E}(t) P_1^0 f\\&
=\sum\limits_{\gamma,\gamma'}b_{E}(\gamma,\gamma',t)c(\gamma')\phi_{\gamma}(x)
\end{align*}
where
\begin{align*}
b_{E}(\gamma,\gamma',t)=\int_{\mathbb{R}^n}\overline{\phi_{\gamma}(x)}\phi_{\gamma'}(t,x)\,dx
\end{align*}
denotes the entries of the matrix $B_{E}(t)$. Then note that $TE(t)$ is defined by the following equation  
\begin{align*}
\Pi^0 TE(t) \Pi^0 f & = P^0_2 B_{T}(t) P_1^0f \\&
=\sum\limits_{\gamma,\gamma'}b_{T}(\gamma,\gamma',t)c(\gamma')\phi_{\gamma}(x)
\end{align*}
where 
\begin{align*}
b_{T}(\gamma,\gamma',t)=\int_{\mathbb{R}^n}\overline{\phi_{\gamma}(x)}T\phi_{\gamma'}(t,x)\,dx
\end{align*}
denotes the entries of the matrix $B_{T}(t)$. The central theorem of this section is then: 
\begin{thm}\label{boundedtheorem}
$E(t)$ is a bounded operator of order $0$, and $TE(t)$ is a bounded operator of order $1$.
\end{thm}

From Section 1, $\Pi^0$ is bounded and invertible and also by Theorem \ref{theoremone} we know the relationship of the frame to the Sobolev norm of $f(x)$. Therefore to prove Theorem 3 by Schur's lemma it suffices to show:
\begin{align}\label{star}
\sum\limits_{\gamma}|b_{E}(\gamma,\gamma',t)|\leq C, \qquad \sum\limits_{\gamma'}|b_{E}(\gamma,\gamma',t)|\leq C 
\end{align}
and
\begin{align}\label{doublestar}
\sum\limits_{\gamma}|b_{T}(\gamma,\gamma',t)|\leq C2^{k'}, \qquad \sum\limits_{\gamma'}|b_{T}(\gamma,\gamma',t)|\leq C2^{k}.
\end{align}
where $C$ denotes a constant independent of $\gamma$ and $\gamma'$. We will also show this constant is uniform for all $t\in [-T,T]$.

We start by examining $T\phi_{\gamma}(t,x)$:
\begin{lem}\label{leading}
\[
T\phi_{\gamma}(t,x)=\sparen{\frac{|\xi_{\gamma}(t)|\Delta x_{\gamma}}{2\pi}}^{\frac{n}{2}}\times \sparen{p(x,t,\psi_x,\psi_t)e^{i\psi}+\mathcal{O}(|\xi_{\gamma}(t)|)e^{i\psi}}
\]
where 
\[
\psi(t,x,x_{\gamma}(t),\xi_{\gamma}(t))=\xi(t)\cdot(x-x_{\gamma}(t))+i|\xi_{\gamma}(t)||x-x_{\gamma}(t)|^2
\]
and
\[
p(t,x,\psi_x,\psi_t)=\mathcal{O}(|\xi_{\gamma}(t)|^2|x-x_{\gamma}(t)|^2).
\]
\end{lem} 
\begin{proof}
As $p(t,x,\psi_x,\psi_t)$ is positive homogeneous of degree two in $|\xi_{\gamma}(t)|$, the desired conclusion will follow if on nul-bicharacteristics $(t,x_{\gamma}(t),\xi_{\gamma}(t))$ we can show that 
\[
\nabla_{x}p(t,x,\psi_x(t,x_{\gamma}(t),\xi_{\gamma}(t)),\psi_t(t,x_{\gamma}(t),\xi_{\gamma}(t)))=0.
\]
Computing $\nabla_{x}p(t,x,\psi_x,\psi_t)$,
\begin{equation}\label{expression}
\frac{\partial}{\partial x_j}p(t,x,\psi_x,\psi_t)=p_{x_j}+p_{\xi_l}\psi_{x_lx_j}+p_{\tau}\psi_{\tau x_j}.
\end{equation}
Dividing (\ref{expression}) by $2q$ and substituting the equations in (\ref{bicharacteristics}) into the right hand side of (\ref{expression}) we obtain
\begin{equation}\label{input}
-\frac{d\xi_j}{dt}+ \frac{dx_l}{dt}\psi_{x_l}\psi_{x_j}+\psi_{t_j}
\end{equation}
As $\psi_{x_j}(t,x_{\gamma}(t),\xi_{\gamma}(t))=\xi_j(t)$, differentiating $\xi_j(t)$ with respect to $t$  we have
\begin{equation}\label{dif}
\frac{d\xi_j}{dt}= \frac{dx_l}{dt}\psi_{x_l}\psi_{x_j}+\psi_{t_j}.
\end{equation}
Substituting (\ref{dif}) into (\ref{input}) implies (\ref{input}) is 0, which happens if and only if (\ref{expression}) vanishes on nul-bicharacteristics. 
\end{proof}
With Lemma \ref{leading} in mind, we consider the entries of the matrices $B_E(t)$ and $B_T(t)$. First we set
\[
\beta^{0}_{\gamma,\gamma'}=\sparen{\frac{|\xi_{\gamma}||\xi_{\gamma'}(t)|\Delta x_{\gamma}\Delta x_{\gamma'}}{(2\pi)^2}}^{\frac{n}{2}}
\]
then 
\begin{align*}
&b_E(\gamma,\gamma',t)\\&= 
\beta_{\gamma,\gamma'}^{0}
\int\limits_{\mathbb{R}^n}\exp\sparen{i(x-x_{\gamma'}(t))\cdot \xi_{\gamma'}(t)-i(x-x_{\gamma})\cdot\xi_{\gamma}-|\xi_{\gamma'}(t)||x-x_{\gamma'}(t)|^2-|\xi_{\gamma}|||x-x_{\gamma}|^2}\,dx
\end{align*}
and to leading order
\begin{align}\label{bT}
&b_{T}(\gamma,\gamma',t)\\&= \nonumber
\beta_{\gamma,\gamma'}^{0}
\int\limits_{\mathbb{R}^n}\exp\sparen{i(x-x_{\gamma'}(t))\cdot \xi_{\gamma'}(t)-i(x-x_{\gamma})\cdot\xi_{\gamma}-|\xi_{\gamma'}(t)||x-x_{\gamma'}(t)|^2-|\xi_{\gamma}||x-x_{\gamma}|^2}\\ \nonumber
 &\times |x-x_{\gamma'}(t)|^2|\xi_{\gamma'}(t)|^2\,dx. 
\end{align}
The first inner product, $b_E(\gamma,\gamma',t)$ is evaluated 
\begin{align*}
&b_E(\gamma,\gamma',t)\\&=
\beta_{\gamma,\gamma'}^{0} 
\int\limits_{\mathbb{R}^n}\exp\sparen{i(x-x_{\gamma'}(t))\cdot \xi_{\gamma'}(t)-i(x-x_{\gamma})\cdot\xi_{\gamma}-|\xi_{\gamma'}(t)||x-x_{\gamma'}(t)|^2-|\xi_{\gamma}|||x-x_{\gamma}|^2}\,dx \\
&=\beta_{\gamma,\gamma'}^{0} \exp\sparen{ix_{\gamma}\cdot\xi_{\gamma}-ix_{\gamma'}(t)\cdot \xi_{\gamma'}(t)}\\ &=
\int\limits_{\mathbb{R}^n}\exp\sparen{ix\cdot(\xi_{\gamma'}(t)-\xi_{\gamma})} \exp\sparen{-\sparen{|\xi_{\gamma}|+|\xi_{\gamma'}(t)|}\sabs{x-\frac{|\xi_{\gamma}|x_{\gamma}+|\xi_{\gamma'}(t)|x_{\gamma'}(t)}{|\xi_{\gamma}|+|\xi_{\gamma'}|}}^2}\\ &\times 
 \exp\sparen{-\frac{|\xi_{\gamma'}(t)||\xi_{\gamma}|}{|\xi_{\gamma}|+|\xi_{\gamma'}(t)|}{|x_{\gamma'}(t)-x_{\gamma}|^2}}\,dx.
\end{align*}
Making the change of variable
\begin{align}\label{sub}
y=x-\frac{|\xi_{\gamma}|x_{\gamma}+|\xi_{\gamma'}(t)|x_{\gamma'}(t)}{|\xi_{\gamma}|+|\xi_{\gamma'}(t)|},
\end{align}
$b_E(\gamma,\gamma',t)$ takes the form of the Fourier transform of a Gaussian integral which one can evaluate (see Appendix 1) to obtain:  
\begin{align*}
&\beta_{\gamma,\gamma'}\exp\sparen{i\sparen{x_{\gamma}\cdot\xi_{\gamma}-x_{\gamma'}(t)\cdot\xi_{\gamma'}(t)+\frac{|\xi_{\gamma}|x_{\gamma}+|\xi_{\gamma'}(t)|x_{\gamma'}(t)}{|\xi_{\gamma}|+|\xi_{\gamma'}(t)|}\cdot(\xi_{\gamma'}(t)-\xi_{\gamma})}} \\ \nonumber
&\times  \exp\sparen{-\frac{\sabs{\xi_{\gamma'}(t)-\xi_{\gamma}}^2}{4(|\xi_{\gamma}|+|\xi_{\gamma'}(t)|)}} \exp\sparen{-\frac{|\xi_{\gamma'}(t)||\xi_{\gamma}|}{|\xi_{\gamma}|+|\xi_{\gamma'}(t)|}\sabs {x_{\gamma'}(t)-x_{\gamma}}^2} 
\end{align*}
where now
\begin{align}
\beta_{\gamma,\gamma'}=\sparen{\frac{|\xi_{\gamma}||\xi_{\gamma'}(t)|\Delta x_{\gamma}\Delta x_{\gamma'}}{4\pi(|\xi_{\gamma'}(t)|+|\xi_{\gamma}|)}}^{\frac{n}{2}}.
\end{align}
so that
\begin{align}\label{bebound}
&|b_E(\gamma,\gamma',t)| \\ \nonumber
&\leq \beta_{\gamma,\gamma'}
\times  \exp\sparen{-\frac{\sabs{\xi_{\gamma'}(t)-\xi_{\gamma}}^2}{4(|\xi_{\gamma}|+|\xi_{\gamma'}(t)|)}} \exp\sparen{-\frac{|\xi_{\gamma'}(t)||\xi_{\gamma}|}{|\xi_{\gamma}|+|\xi_{\gamma'}(t)|}\sabs {x_{\gamma'}(t)-x_{\gamma}}^2}. 
\end{align}

For the integral $b_T(\gamma,\gamma',t)$ we make the same substitution (\ref{sub}) into (\ref{bT}). Then we set 
\[
\eta=\xi_{\gamma'}(t)-\xi_{\gamma},\qquad c=|\xi_{\gamma}|+|\xi_{\gamma'}(t)|
\]
and 
\[
b=\frac{|\xi_{\gamma}|(x_{\gamma}-x_{\gamma'}(t))}{|\xi_{\gamma}|+|\xi_{\gamma'}(t)|}
\]
so that  we can apply the estimates in Appendix 1 which gives $|b_T(\gamma,\gamma',t)|$ is equal  
\begin{align*}
\beta_{\gamma,\gamma'}|\xi_{\gamma'}|^2\exp\sparen{-\frac{\eta^2}{4c}}\sabs{-\frac{\eta^2}{4c^2}+\frac{ib\eta}{c}+b^2+\frac{1}{2c}}.
\end{align*}
Applying Cauchy-Schwartz and back substituting values for $\eta, c,$ and $b$ so that for $C$ a constant independent of $\gamma,\gamma'$
\begin{align*}
&|b_T(\gamma,\gamma',t)|\leq  C\beta_{\gamma,\gamma'}
\sparen{\frac{|\xi_{\gamma'}(t)|^2|\xi_{\gamma}-\xi_{\gamma'}(t)|^2}{(|\xi_{\gamma'}(t)|+|\xi_{\gamma}|)^2} + \frac{|\xi_{\gamma'}(t)|^2|\xi_{\gamma}|^2|x_{\gamma}-x_{\gamma'}(t)|^2}{(|\xi_{\gamma}|+|\xi_{\gamma'}(t)|)^2}} \\
&\times 
\exp\sparen{-\frac{\sabs{\xi_{\gamma'}(t)-\xi_{\gamma}}^2}{4(|\xi_{\gamma}|+|\xi_{\gamma'}(t)|)}}
\exp\sparen{-\frac{|\xi_{\gamma'}(t)||\xi_{\gamma}|}{|\xi_{\gamma}|+|\xi_{\gamma'}(t)|}
\sabs{x_{\gamma'}(t)-x_{\gamma}}^2} 
\\
&\leq 
C\beta_{\gamma,\gamma'}
\sparen{|\xi_{\gamma}-\xi_{\gamma'}(t)|^2+ |\xi_{\gamma'}(t)||\xi_{\gamma}||x_{\gamma}-x_{\gamma'}(t)|^2}\\
&\times 
\exp\sparen{-\frac{\sabs{\xi_{\gamma'}(t)-\xi_{\gamma}}^2}{4(|\xi_{\gamma}|+|\xi_{\gamma'}(t)|)}}
\exp\sparen{-\frac{|\xi_{\gamma'}(t)||\xi_{\gamma}|}{|\xi_{\gamma}|+|\xi_{\gamma'}(t)|}
\sabs{x_{\gamma'}(t)-x_{\gamma}}^2} 
\end{align*} 

The next two Lemmas characterize properties of the evolution operator $U(t,t')$ acting on the lattice which will assist us in obtaining the bounds (\ref{star}) and (\ref{doublestar}).
\begin{lem}\label{propertyA}
As in the introduction we let 
\[
A(x,t)=\{a_{ij}(x,t)\}_{1\leq i,j\leq n}
\]
be a real symmetric $n\times n$ matrix with entries $a_{ij}(x,t)$ in $C^{1,1}$. For the rest of this lemma $C>0$ denotes a constant which is independent of the essential variables.  Furthermore, as before $A(t,x)$ is bounded and positive definite  and we let
 \[
 q(x,t,\xi)=\sparen{\sum\limits_{i,j}a_{ij}(x,t)\xi_i\xi_j}^{\frac{1}{2}}
 \]
If we consider the system 
 \begin{equation}\label{system}
 \frac{dx}{dt}=-q_{\xi} \qquad \frac{d\xi}{dt}=q_x
 \end{equation}
 with initial conditions $|x(0)|<R$ and $\frac{1}{a}<|\xi(0)|<a$ for some finite $a,R>0$ then solutions to the system (\ref{system}) satisfy the following two conditions. 
 \begin{enumerate}
\item
$ |x(t)-x(0)|<C\sqrt{n}|T|$ 
\item
For all finite $T>0$ there exists a constant $C(T,a)$ such that
\[
\frac{1}{C(T,a)}<|\omega(t)|<C(T,a)
\]
whenever $|t|<T$
\end{enumerate}
\end{lem}

\begin{proof} We prove condition (1) first and then condition (2).
\begin{enumerate}
\item 
Computing $q_{\xi_i}$
\begin{equation}\label{one}
\frac{\partial}{\partial\xi_i}\sparen{\sum\limits_{ij}a_{ij}(x,t)\xi_i\xi_j}^{\frac{1}{2}}=\frac{\sum\limits_{j}a_{ij}(x,t)\xi_j}{\sparen{\sum\limits_{i,j}a_{ij}(x,t)\xi_i\xi_j}^{\frac{1}{2}}}<C
\end{equation}
since the expression in the middle of (\ref{one}) is homogeneous of degree $0$ and the numerator and denominator are both bounded above and below on $|\xi|=1$. From (\ref{system}) we then have 
\[
\sabs{\frac{dx_i}{dt}}<C 
\]
which implies
\[
\sabs{\frac{dx}{dt}}<C\sqrt{n}.
\]
Integrating this inequality gives (1).
\item
Differentiating $q$ with respect to $x$ we have,
\begin{equation}\label{two}
\sabs{\frac{d\xi}{dt}}=\sabs{\frac{\sum\limits_{ij}(a_{ij})_x(x,t)\xi_i\xi_j}{2\sparen{\sum\limits_{i,j}a_{i,j}(x,t)\xi_i\xi_j}^{\frac{1}{2}}}}<C|\xi|
\end{equation}
for some C independent of $\xi$, $x$, and $t$, since the expression in the middle of (\ref{two}) is homogeneous of degree $1$ and the numerator and denominator are both bounded above and below on $|\xi|=1$
Using Gronwall's inequality  gives 
\begin{equation}\label{bounds}
\frac{1}{C(T,a)}<|\xi(0)|\exp(-Ct)<|\xi(t)|<|\xi(0)|\exp(Ct)<C(T,a)
\end{equation}
This results in the desired conclusion for finite $T$, that is for all $t\in [-T,T]$ if $\frac{1}{a}<|\xi(0)|<a$ then there exists $C(T,a)$ such that condition (\ref{bounds}) holds. 
\end{enumerate}
\end{proof}

Recall that $U(t,t')$ is the evolution operator associated to (\ref{system}) and
\[
U(t,t')(x_{\gamma},\xi_{\gamma})=(x_{\gamma}(t,t',x_{\gamma},\xi_{\gamma}),\xi_{\gamma}(t,t',x_{\gamma},\xi_{\gamma})).
\]
By homogeneity, if $c$ is a constant then the above equation scales as follows 
\begin{align}\label{scaling}
(x_{\gamma}(t,t',x_{\gamma},\xi_{\gamma}),c\xi_{\gamma}(t,t',x_{\gamma},\xi_{\gamma}))=
(x_{\gamma}(t,t',x_{\gamma},c\xi_{\gamma}),\xi_{\gamma}(t,t',x_{\gamma},c\xi_{\gamma})).
\end{align}
Since $\xi_{\gamma}=2^k\omega_{i,k}$ with $\frac{1}{2}\leq |\omega_{i,k}|<1$ the relationship (\ref{scaling}) with $c=2^{-k}$ gives that the pair $(x_{\gamma},\omega_{\gamma})$ lies in a compact subset of $\mathbb{R}^n\times(\mathbb{R}^n/\{0\})$ whenever $|x_{\gamma}|<R$ with $R$ a constant independent of $\gamma$. We note that Lemma \ref{propertyA} then applies to $(x_{\gamma},\omega_{\gamma})$ so we that have a bound on the size of $x_{\gamma}(t)$ and $\xi_{\gamma}=2^k\omega_{i,k}(t)$ in terms of the initial data. 

Because our frame is similar to an almost orthogonal frame in type, it makes sense that the pairs of initial data which are close together in frequency contribute the most to the absolute value of the inner products in the sums in (\ref{star}) and (\ref{doublestar}). However, we have an extra variable $\alpha$ since we have a non-compactly supported set of frame functions. Therefore we will use the term 'close in frequency' to mean that the pairs $(x_{\gamma},\xi_{\gamma})$ and $(x_{\gamma'},\xi_{\gamma'})$ from Section 1 satisfy not only the condition  $|x_{\gamma}|, |x_{\gamma'}|<R$ but also that $|k-k'|\leq k_0$, where $k_0$ is a finite constant independent of $\gamma,\gamma'$.  We will show that close pairs of lattice variables have an extra property (Lemma \ref{propertyB}) beyond that of Lemma \ref{propertyA} which makes it possible to compute the bounds on (\ref{star}) and (\ref{doublestar}). 

First we see by equation (\ref{scaling}), for all such close pairs with $c=2^{-k'}$ (where here without loss of generality we have taken $k'\leq k$) the corresponding scaled pairs $(x_{\gamma},2^{k-k'}\omega_{\gamma})$ and $(x_{\gamma'},\omega_{\gamma'})$ lie in the same compact subset $[-R,R]^n\times [\frac{1}{2},2^{k_0}]^n$ of $\mathbb{R}^n\times(\mathbb{R}^n/\{0\})$.  Thus we can conclude from Lemma \ref{propertyA} that the transformation $\mathcal{U}(t,0)$ is invertible and Lipschitz with uniform Lipschitz constant when acting on $(x_{\gamma},2^{k-k'}\omega_{\gamma})$ and $(x_{\gamma'},\omega_{\gamma'})$.  In other words, for all close pairs and for all $t\in[-T,T]$ with $T$ fixed and finite there exists nonzero constants $D_1$ and $D_2$ independent of $\gamma,\gamma'$ with
\begin{align}\label{Lip}
& D_1d^2((x_{\gamma}(t),2^{k-k'}\omega_{\gamma}(t)); (x_{\gamma'}(t),\omega_{\gamma'}(t))) 
\leq d^2((x_{\gamma},2^{k-k'}\omega_{\gamma});(x_{\gamma'},\omega_{\gamma'})) \\ & \leq \nonumber D_2d^2((x_{\gamma}(t),2^{k-k'}\omega_{\gamma}(t)); (x_{\gamma'}(t),\omega_{\gamma'}(t)))
\end{align}
where $d$ denotes the usual Euclidean distance. For convenience, we will abbreviate this type of equivalence relationship where the left hand side is bounded above and below by multiples of the right hand side by $\sim$, so that inequality (\ref{Lip}) can be rewritten as
\begin{align*}
d^2((x_{\gamma},2^{k-k'}\omega_{\gamma});(x_{\gamma'},\omega_{\gamma'}))\sim d^2((x_{\gamma}(t),2^{k-k'}\omega_{\gamma}(t)); (x_{\gamma'}(t),\omega_{\gamma'}(t))).
\end{align*}

The inequality (\ref{Lip}) allows us to obtain another, similar relationship which is the crucial relationship in the computations to obtain bounds on the action of the matrices $B_E(t)$ and $B_T(t)$.  
\begin{lem}\label{propertyB}
For pairs $(x_{\gamma},\xi_{\gamma})$ and $(x_{\gamma'},\xi_{\gamma'})$ such that $|x_{\gamma}|, |x_{\gamma'}|<R$ where $0<R<\infty$ and $|k-k'|\leq k_0$ with $R$ and $k_0$ independent of $\gamma$ and $\gamma'$ the following holds
\[
d^2(\mathcal{U}(t,0)(x_{\gamma'},\omega_{\gamma'}); (x_{\gamma},2^{k-k'}\omega_{\gamma}))\sim d^2(\mathcal{U}(0,t)(x_{\gamma},2^{k-k'}\omega_{\gamma});(x_{\gamma'},\omega_{\gamma'}))
\]
\end{lem}
\begin{proof} Since $\mathcal{U}(t,0)\circ \mathcal{U}(0,t)=I$, the right hand side of the relationship, 
\[
d^2(\mathcal{U}(0,t)(x_{\gamma},2^{k-k'}\omega_{\gamma});(x_{\gamma'},\omega_{\gamma'}))
\]
can be expressed as
\[
d^2(\mathcal{U}(0,t)(x_{\gamma},2^{k-k'}\omega_{\gamma});\mathcal{U}(0,t)\mathcal{U}(t,0)(x_{\gamma'},\omega_{\gamma'}))
\]
and from estimate (\ref{Lip}) we obtain the desired conclusion.
\end{proof}

With these Lemmas, we can now calculate a bound on
\begin{equation}\label{bigE}
\sum\limits_{\gamma'}|b_E(\gamma,\gamma',t)|
\end{equation}
for fixed $\gamma$.
We break this sum into three pieces: in region 1. $\gamma': k'<k-k_0$, in region 2, $\gamma':|k'-k|\leq  k_0$ and in region 3, $\gamma': k'>k+k_0$ where for all $t\in [-T,T], T<\infty$, $k_0=\max\{2\log_2C(T,a),1\}$. For the rest of this argument let $D>0$ denote a constant which is independent of $k'$ and $k$ and is uniform for all $t\in [-T,T]$. 

We will apply Lemma \ref{propertyA} in each region to subsets of the initial data $(x_{\gamma'},\omega_{\gamma'})$ as outlined earlier.  Here we must cutoff the $x_{\gamma'}$'s so that $|x_{\gamma'}|<R$, for some large positive $R$. This corresponds to having the initial data with support living in a ball of radius $R$. 

 %(By construction, $\frac{1}{2}\leq|\omega_{\gamma'}|<1$ always.) The cutoffs introduce only a very small error, since for all $f(x)\in L^2(\mathbb{R}^n)$ and for all small $\epsilon>0$ there exists $R>0$ such that $\norm{f(x)\chi_{|x|>R}}_{L^2(\mathbb{R}^n)}<\epsilon$. 
%We pick $\epsilon$  accordingly small (or $0$ in the case where $f(x)$ has compact support).  The frame coefficients corresponding to those $x_{\gamma}$ contained in the set $\{|x|>R\}$ will be correspondingly small and as $k$ increases they will decrease until they are identically $0$.

Again because of the similarity of the frame to an almost orthogonal frame, in regions 1 and 3  from Lemma \ref{propertyA} the exponential term from the bound on each inner product will dominate the sum, but in region 2 the argument is more subtle. In each case, formula (\ref{scaling}) and Lemma \ref{propertyA} imply that: 
\begin{align}\label{factor}
&\beta_{\gamma,\gamma'}=\sparen{\frac{|\xi_{\gamma}||\xi_{\gamma'}(t)|\Delta x_{\gamma}\Delta x_{\gamma'}}{4\pi(|\xi_{\gamma}|+|\xi_{\gamma'}|)}}^{\frac{n}{2}}=
\sparen{\frac{C_{\epsilon}^22^{\frac{k}{2}-\epsilon k}2^{\frac{k'}{2}-\epsilon k'}|\omega_{\gamma}||\omega_{\gamma'}(t)|}{4\pi(2^k|\omega_{\gamma}|+2^{k'}|\omega_{\gamma'}(t)|)}}^{\frac{n}{2}} \\ &\leq \nonumber
 \sparen{\frac{C^2_{\epsilon}2^{\frac{k}{2}-\epsilon k}2^{\frac{k'}{2}-\epsilon k'}C(T,a)}{4\pi(2^{k-1}+\frac{2^{k'}}{C(T,a)})}}^{\frac{n}{2}}
 \end{align}
The right hand side of  (\ref{bebound}) contains a product of two exponentials with arguments
\begin{align}\label{expxi}
-\frac{|2^k\omega_{\gamma}-2^{k'}\omega_{\gamma'}(t)|^2}{4(2^k|\omega_{\gamma}|+2^{k'}|\omega_{\gamma'}(t)|)}
\end{align}
and
\begin{align}\label{expx}
-\frac{|\xi_{\gamma}||\xi_{\gamma'}(t)|}{|\xi_{\gamma}|+|\xi_{\gamma'}(t)|}|x_{\gamma}-x_{\gamma'}(t)|^2
\end{align}

In region 1, Lemma \ref{propertyA} implies a lower bound on (\ref{factor})
\begin{align*}
&\frac{|2^k\omega_{\gamma}-2^{k'}\omega_{\gamma'}(t)|^2}{4(2^k|\omega_{\gamma}|+2^{k'}|\omega_{\gamma'}(t)|)}>
\frac{\sparen{2^{k-1}-2^{k'}C(T,a)}^2}{4(2^k+2^{k'}C(T,a))}\\&>
\frac{2^{k-5}\sparen{1-2^{k'-k+1}C(T,a)}^2}{(1+2^{k'-k}C(T,a))}>
\frac{2^{k-5}\sparen{1-\frac{2}{C(T,a)}}^2}{\sparen{1+\frac{1}{C(T,a)}}}=2^{k}D.
\end{align*}
For (\ref{expx}), we only know that 
\[
|x_{\gamma}-x_{\gamma'}(t)|^2\geq 0
\]
which gives 
\[
\exp\sparen{-|x_{\gamma}-x_{\gamma'}(t)|^2}\leq 1.
\]
Since by assumption $|x_{\gamma'}|<R$ and $x_{\gamma'}=\Delta x_{\gamma'}\alpha'=C_{\epsilon}2^{-\frac{k'}{2}-\epsilon k}\alpha'$ for fixed $k'$, by scaling we have $|\alpha'|<RC_{\epsilon}2^{\frac{k'}{2}+\epsilon k'}$. Bounding the number of points in $\mathbb{Z}^n$ in this ball by $D(2^{\frac{k'}{2}+\epsilon k'})^{n}$, we obtain a bound on the number of $x_{\gamma'}$ for fixed $(i,k')$. While the position of the $x_{\gamma'}$ may change their total number does not change when they are propagated. From Lemma \ref{lemmaone} there are $\mathcal{O}(2^{\frac{k'n}{2}})$ vectors $\omega_{i,k'}$ in each annulus indexed by $k'$.  In region 1, from (\ref{factor}) we find since $k'<k-k_0$
\begin{align*}
&\beta_{\gamma,\gamma'}\leq 
 \sparen{\frac{C_{\epsilon}^22^{\frac{k}{2}-\epsilon k}2^{\frac{k'}{2}-\epsilon k'}C(T,a)}{4\pi(2^{k-1}+\frac{2^{k'}}{C(T,a)})}}^{\frac{n}{2}}\leq D\sparen{2^{-\frac{k}{2}-\epsilon k}2^{\frac{k'}{2}-\epsilon k'}}^{\frac{n}{2}}
\end{align*}

Combining estimates
\begin{align}\label{sumregionone}
&\sum\limits_{\gamma': k'<k-k_0}|b_E(\gamma,\gamma',t)|=\mathcal{O}\sparen{\sum\limits_{\substack{(i,k')\\ k'<k-k_0}}2^{-\frac{nk}{4}-\frac{\epsilon nk}{2}}2^{\frac{3nk'}{4}+\frac{\epsilon nk'}{2}}\exp\sparen{-D2^{k}}}\\=&  \nonumber
\mathcal{O}\sparen{\sum\limits_{k'<k-k_0}2^{nk'}\exp\sparen{-D2^k}}= \nonumber\mathcal{O}\sparen{2^{nk}\exp\sparen{-D2^k}}
\end{align}
But since $2^{nk}\exp\sparen{-D2^{k}}\rightarrow 0$ as $k\rightarrow \infty$, the contribution from (\ref{sumregionone}) is bounded  independent of $\gamma, \gamma'$.

Similarly, in region 3 an application of Lemma \ref{propertyA} to the first part of the exponential contribution (\ref{expxi}) gives:  
\begin{align*}
&\frac{|2^k\omega_{\gamma}(0)-2^{k'}\omega_{\gamma'}(t)|^2}{4(2^k|\omega_{\gamma}(0)|+2^{k'}|\omega_{\gamma'}(t)|)}>\frac{\sparen{\frac{2^{k'}}{C(T,a)}-2^k}^2}{4(2^k+2^{k'}C(T,a))}\\&>\frac{2^{k'-3}}{C(T,a)^3}-\frac{2^{k-2}}{C(T,a)^2}>D2^{k'}.
\end{align*}
Again, the same estimates as in region 1 for the number of the $x_{\gamma'}(t)$ and their exponential contribution hold, and the number of  vectors $\omega_{i,k'}$ in each annulus for fixed $k'$ is still $\mathcal{O}(2^{\frac{nk'}{2}})$. For the size of $\beta_{\gamma,\gamma'}$ from (\ref{factor}) and the fact $k'>k+k_0$, we have
\begin{align*}
&\beta_{\gamma,\gamma'} \leq \nonumber
 \sparen{\frac{C_{\epsilon}^22^{\frac{k}{2}-\epsilon k}2^{\frac{k'}{2}-\epsilon k'}C(T,a)}{4\pi(2^{k-1}+\frac{2^{k'}}{C(T,a)})}}^{\frac{n}{2}}<D\sparen{2^{\frac{k}{2}-\epsilon k}2^{-\frac{k'}{2}-\epsilon k'}}^{\frac{n}{2}}.
\end{align*}
Thus
\begin{align}\label{sumregionthree}
&\sum\limits_{\gamma': k'>k+k_0}|b_E(\gamma,\gamma',t)|=\mathcal{O}\sparen{\sum\limits_{\substack{(i,k')\\ k'>k+k_0}}2^{\frac{nk}{4}-\frac{\epsilon nk}{2}}2^{\frac{nk'}{4}+\frac{\epsilon nk'}{2}}\exp\sparen{-D2^{k}}}\\&=  \nonumber
\mathcal{O}\sparen{\sum\limits_{k'>k+k_0}2^{\frac{3nk'}{4}+\frac{kn}{4}+\frac{\epsilon n(k'-k)}{2}}\exp\sparen{-D2^{k'}}}.
\end{align}
By hypothesis, $k'>k+k_0$, so the exponential term dominates the sum here as well so the contribution from (\ref{sumregionthree}) is bounded independent of $\gamma, \gamma'$. 

If we try to simply apply Lemma \ref{propertyA} in region 2, as we did in regions 1 and 3 we get a constant bound on the exponential contributions (\ref{expxi}) and (\ref{expx}) which is not enough to dominate the contributions to the sum from the number of $x_{\gamma}$ and $\xi_{\gamma}$. Therefore the application of Lemma \ref{propertyB} to the exponential term is essential in region 2 since the treatment of the exponential contribution to the summation is more delicate there. The key is that the additional Lemma \ref{propertyB} allows us to sum over unpropagated variables which we have information about given from the construction that they are fixed in space. 

In region 2, by homogeneity and the fact $|k-k'|\leq k_0$ the entire exponential term can be re-written as follows 
\begin{align}\label{hom1}
&\frac{\sabs{\xi_{\gamma'}(t)-\xi_{\gamma}}^2}{4(|\xi_{\gamma}|+|\xi_{\gamma'}(t)|)}+ \frac{|\xi_{\gamma'}(t)||\xi_{\gamma}|}{|\xi_{\gamma}|+|\xi_{\gamma'}(t)|}
\sabs{x_{\gamma'}(t)-x_{\gamma}}^2 \\ \nonumber \sim& 2^{k'}d^2(\mathcal{U}(t,0)(x_{\gamma},2^{k-k'}\omega_{\gamma});(x_{\gamma'},\omega_{\gamma'}))
\end{align}
Applying Lemma (\ref{propertyB}) to (\ref{hom1}) we obtain
\begin{align}\label{hom2}
&\frac{\sabs{\xi_{\gamma'}(t)-\xi_{\gamma}}^2}{4(|\xi_{\gamma}|+|\xi_{\gamma'}(t)|)}+ \frac{|\xi_{\gamma'}(t)||\xi_{\gamma}|}{|\xi_{\gamma}|+|\xi_{\gamma'}(t)|}
\sabs{x_{\gamma'}(t)-x_{\gamma}}^2 \\ \nonumber \sim & 2^{k'}d^2((x_{\gamma},2^{k-k'}\omega_{\gamma});\mathcal{U}(0,t)(x_{\gamma'},\omega_{\gamma'}))
\end{align}
where the constants in this equivalence relation may depend on $k_0$ as well but are uniform in $T$. 

Since $x_{\gamma'}=\Delta x_{\gamma'}\alpha'$ we can factor out the scaling $\Delta x_{\gamma'}=C_{\epsilon}2^{-\frac{k'}{2}-\epsilon k'}$ from part of the right hand side of (\ref{hom2})
\[
2^{k'}\sabs{x_{\gamma'}-x_{\gamma}(-t)}^2=C_{\epsilon}^{2}2^{-2\epsilon k'}\sabs{\alpha'-\epsilon_0}^2
\]
where we have set $\epsilon_0=(\Delta x_{\gamma'})^{-1}x_{\gamma}(-t)$. Substituting $\lambda=C_{\epsilon}^{-2}2^{2\epsilon k'}$, an application of the integral estimates in the appendix gives
\begin{align}\label{xband}
&\sum\limits_{\alpha'\in\mathbb{Z}^n,|\alpha'|<R}\exp\sparen{-2^{k'}\sabs{x_{\gamma'}-x_{\gamma}(-t)}^2}<\sum\limits_{\alpha'\in\mathbb{Z}^n}\exp\sparen{-C^2_{\epsilon}2^{-2\epsilon k'}\sabs{\alpha'-\epsilon_0}^2}\\&=\nonumber
\mathcal{O}\sparen{2^{\epsilon k'n}}
\end{align}
But looking at inequality (\ref{factor}) the $k'$ dependence in this last bound is exactly canceled by the size of $\beta_{\gamma,\gamma'}$ in region 2. 
Since $|k-k'|\leq k_0$, the other part of the exponential contribution we tackle with an argument similar to that of Lemma \ref{lemmaone} applied to $\xi=2^k\omega_{\gamma}(-t)$ which implies the sum 
\begin{align}\label{band}
\sum\limits_{\substack{(i,k')\\ |k'-k|\leq k_0}}\exp\sparen{2^{-k'}\sabs{2^{k'}\omega_{\gamma'}-2^k\omega_{\gamma}(-t)}^2}
\end{align}
is bounded independent of $\gamma,\gamma'$. From these bounds and from the equivalence relation (\ref{hom2}) we can conclude   
\begin{align}\label{sumregiontwo}
\sum\limits_{\gamma': |k'-k|\leq k_0}|b_E(\gamma,\gamma',t)|=\mathcal{O}\sparen{1}
\end{align}
and thus 
\begin{align*}
\sum\limits_{\gamma'}|b_E(\gamma,\gamma',t)|=\mathcal{O}\sparen{1}
\end{align*}

If we reverse the roles of $\gamma$ and $\gamma'$ we can run a similar argument to the one above to bound 
\[
\sum\limits_{\gamma}|b_E(\gamma,\gamma',t)|
\]
The bounds in each of the regions $|k-k'|\leq k_0$, $k>k'+k_0$ and $k<k'-k_0$ will follow almost identically. The main difference in the argument will be the region where $|k-k'|\leq k_0$ we do not need to apply Lemma \ref{propertyB} since the $\gamma$ variables are not propagated. In this way we obtain the desired bound (\ref{star}).
 
For the estimate (\ref{doublestar}), we examine 
\begin{equation}\label{bigT}
\sum\limits_{\gamma'}|b_T(\gamma,\gamma',t)|
\end{equation}
The only difference between the bound on $|b_E(\gamma,\gamma',t)|$ and that of $|b_T(\gamma,\gamma',t)|$ is the factor of  $|\xi_{\gamma}-\xi_{\gamma'}(t)|^2+ |\xi_{\gamma'}(t)||\xi_{\gamma}||x_{\gamma}-x_{\gamma'}(t)|^2$. In region 1, application of Lemma \ref{propertyA} gives
\begin{align*}
&|\xi_{\gamma}-\xi_{\gamma'}(t)|^2+ |\xi_{\gamma'}(t)||\xi_{\gamma}||x_{\gamma}-x_{\gamma'}(t)|^2\\ &\leq (|\xi_{\gamma}|+|\xi_{\gamma'}(t)|)^2+ |\xi_{\gamma'}(t)||\xi_{\gamma}|(|x_{\gamma}-x_{\gamma'}|+|x_{\gamma'}(t)-x_{\gamma'}|)^2 \\ &\leq 
\sabs{2^k+2^{k'}C(T,a)}^2+2^{k}2^{k'}C(T,a)(R+T)^2
\leq D2^{2k}
\end{align*}
The rest of the estimates on $\beta_{\gamma,\gamma'}$ and the exponential contribution stay the same. Therefore by (\ref{sumregionone})
\begin{align*}
&\sum\limits_{\gamma': k'<k-k_0}|b_T(\gamma,\gamma',t)|=\mathcal{O}\sparen{\sum\limits_{\substack{(i,k')\\ k'<k-k_0}}2^{2k}2^{-\frac{nk}{4}-\frac{\epsilon nk}{2}}2^{\frac{3nk'}{4}+\frac{\epsilon nk'}{2}}\exp\sparen{-D2^{k}}}\\&=  \nonumber
\mathcal{O}\sparen{\sum\limits_{k'<k-k_0}2^{2k}2^{nk'}\exp\sparen{-D2^k}}= \nonumber\mathcal{O}\sparen{2^{(n+2)k}\exp\sparen{-D2^k}}
\end{align*}
and as $k\rightarrow \infty$, $2^{(n+2)k}\exp(-D2^{k})\rightarrow 0$ so the sum in question is also uniformly bounded independent of $\gamma, \gamma'$

Similarly, in region 3, by Lemma \ref{propertyA} the extra factor is bounded by: 
\begin{align*}
&|\xi_{\gamma}-\xi_{\gamma'}(t)|^2+ |\xi_{\gamma'}(t)||\xi_{\gamma}||x_{\gamma}-x_{\gamma'}(t)|^2\\ &\leq (|\xi_{\gamma}|+|\xi_{\gamma'}(t)|)^2+ |\xi_{\gamma'}(t)||\xi_{\gamma}|(|x_{\gamma}-x_{\gamma'}|+|x_{\gamma'}(t)-x_{\gamma'}|)^2 \\ &\leq 
\sabs{2^k+2^{k'}C(T,a)}^2+2^{k}2^{k'}C(T,a)(R+T)^2 
\leq D2^{2k'}
\end{align*}
Again, the rest of the estimates stay the same, so that analogous to (\ref{sumregionthree})
\begin{align}
&\sum\limits_{\gamma': k'>k+k_0}|b_T(\gamma,\gamma',t)|=\mathcal{O}\sparen{\sum\limits_{\substack{(i,k')\\ k'>k-+k_0}}2^{2k'}2^{\frac{nk}{4}-\frac{\epsilon nk}{2}}2^{\frac{nk'}{4}+\frac{\epsilon nk'}{2}}\exp\sparen{-D2^{k}}}\\&=  \nonumber
\mathcal{O}\sparen{\sum\limits_{k'>k+k_0}2^{2k'}2^{\frac{3nk'}{4}+\frac{kn}{4}+\frac{\epsilon n(k'-k)}{2}}\exp\sparen{-D2^{k'}}}
\end{align}
and as before this sum also converges independent of $\gamma,\gamma'$ since $k'>k+k_0$.

The only region where the extra factor in question makes a difference is in region 2. As in the treatment of the sum of $|b_E(\gamma,\gamma')|$ over $\gamma'$,  Lemma \ref{propertyB} is again crucial. By homogeneity and Lemma \ref{propertyB} the extra factor in the bounds for  $|b_T(\gamma,\gamma',t)|$ can be rewritten as
\[
|\xi_{\gamma}-\xi_{\gamma'}(t)|^2+ |\xi_{\gamma'}(t)||\xi_{\gamma}||x_{\gamma}-x_{\gamma'}(t)|^2 \sim 2^{2k'}d^2(\mathcal{U}'(0,t)(x_{\gamma},2^{k-k'}\omega_{\gamma});(x_{\gamma'},\omega_{\gamma'}))
\]
and exponential factor in the bounds still follows the equivalence relation (\ref{hom2}). With these bounds in mind we split the sum
\begin{align}\label{bTone}
&\sum\limits_{\gamma':|k-k'|\leq k_0}\beta_{\gamma,\gamma'}(|\xi_{\gamma}-\xi_{\gamma'}(t)|^2+ |\xi_{\gamma'}(t)||\xi_{\gamma}||x_{\gamma}-x_{\gamma'}(t)|^2)\\ &\times \nonumber
\exp\sparen{-\frac{|2^k\omega_{\gamma}-2^{k'}\omega_{\gamma'}(t)|^2}{4(2^k|\omega_{\gamma}|+2^{k'}|\omega_{\gamma'}(t)|)}-\frac{|\xi_{\gamma}||\xi_{\gamma'}(t)|}{|\xi_{\gamma}|+|\xi_{\gamma'}(t)|}|x_{\gamma}-x_{\gamma'}(t)|^2}
\end{align}
into two equivalent pieces. The sum (\ref{bTone}) becomes 
\begin{align}\label{pieceoneT}
&\sum\limits_{\gamma':|k-k'|\leq k_0}\beta_{\gamma,\gamma'}2^{2k'}\sabs{x_{\gamma'}-x_{\gamma}(-t)}^2\exp\sparen{-2^{k'}\sabs{x_{\gamma'}-x_{\gamma}(-t)}^2} \\& \times \nonumber
\exp\sparen{2^{-k'}\sabs{2^{k'}\omega_{\gamma'}-2^{k}\omega_{\gamma}(-t)}^2}
\end{align}
and
\begin{align}\label{piecetwoT}
&\sum\limits_{\gamma':|k-k'|\leq k_0}\beta_{\gamma,\gamma'}\sabs{2^{k'}\omega_{\gamma'}-2^k\omega_{\gamma}(-t)}^2\exp\sparen{-2^{k'}\sabs{x_{\gamma'}-x_{\gamma}(-t)}^2}\\& \times \nonumber \exp\sparen{2^{-k'}\sabs{2^{k'}\omega_{\gamma'}-2^{k}\omega_{\gamma}(-t)}^2}.
\end{align}
To handle the sum (\ref{pieceoneT}), again since $x_{\gamma'}=\Delta x_{\gamma'}\alpha'$ we can factor out the scaling $\Delta x_{\gamma'}=C_{\epsilon}2^{-\frac{k'}{2}-\epsilon k'}$ from part of the right hand side of (\ref{hom2})
\[
2^{k'}\sabs{x_{\gamma'}-x_{\gamma}(-t)}^2=C_{\epsilon}^{2}2^{-2\epsilon k'}\sabs{\alpha'-\epsilon_0}^2
\]
and also from the new multiplying factor in (\ref{pieceoneT})
\[
2^{2k'}\sabs{x_{\gamma'}-x_{\gamma}(-t)}^2=C_{\epsilon}^{2}2^{-2\epsilon k'+k'}\sabs{\alpha'-\epsilon_0}^2
\]
where we have set $\epsilon_0=(\Delta x_{\gamma'})^{-1}x_{\gamma}(-t)$. Substituing $\lambda=C_{\epsilon}^{-2}2^{2\epsilon k'}$, an application of the second integral estimate in Appendix 2 gives
\begin{align}\label{xbandtwo}
&\sum\limits_{\alpha'\in\mathbb{Z}^n,|\alpha'|<R}\sparen{2^{2k'}\sabs{x_{\gamma'}-x_{\gamma}(-t)}^2}\exp\sparen{-2^{k'}\sabs{x_{\gamma'}-x_{\gamma}(-t)}^2}\\ &<  \nonumber
\sum\limits_{\alpha'\in\mathbb{Z}^n}\sparen{2^{-2\epsilon k'+k'}\sabs{\alpha'-\epsilon_0}^2}\exp\sparen{-2^{-2\epsilon k'}\sabs{\alpha'-\epsilon_0}^2}\\ &= \nonumber
\mathcal{O}\sparen{2^{k'}(2^{2\epsilon k'})^{\frac{n}{2}}}.
\end{align}
Using the previous estimates (\ref{factor}) and (\ref{band}) and the fact $|k-k'|\leq k_0$ the sum (\ref{pieceoneT}) is $\mathcal{O}(2^{k})$. For the second sum (\ref{piecetwoT}), estimate (\ref{xband}) still applies for the sum over $\alpha'$ so we are reduced to examining 
\begin{align}\label{expcontT}
\sum\limits_{\substack{(i,k')\\ |k'-k|\leq k_0}}\sabs{2^{k'}\omega_{\gamma'}-2^k\omega_{\gamma}(-t)}^2\exp\sparen{2^{-k'}\sabs{2^{k'}\omega_{\gamma'}-2^{k}\omega_{\gamma}(-t)}^2}.
\end{align}
Now if we consider the same sets defined in Lemma \ref{lemmaone} with $\xi=2^k\omega_{\gamma}(-t)$ for the first set $\mathcal{A}$ we get 
\[
\sabs{2^{k'}\omega_{\gamma'}-2^k\omega_{\gamma}(-t)}^2\leq 2^k
\]
implying from previous bounds on the number of $\omega_{\gamma'}$ in $\mathcal{A}$: 
\begin{align*}
&\sum\limits_{\mathcal{A}}\sabs{2^{k'}\omega_{\gamma'}-2^k\omega_{\gamma}(-t)}^2\exp\sparen{2^{-k'}\sabs{2^{k'}\omega_{\gamma'}-2^k\omega_{\gamma}(-t)}^2}\leq 3^{n+1}2^k \\ 
&\leq  2^{\frac{k}{2}}j+C(T,a)2^k\leq D2^k
\end{align*}
In each of the sets $\mathcal{B}_j$ 
\[
\sabs{2^{k'}\omega_{\gamma'}-2^k\omega_{\gamma}(-t)}^2\leq j^22^k
\]
and similarly from an argument in Lemma 1 we can deduce 
\begin{align*}
&\sum\limits_{\mathcal{B}}\sabs{2^{k'}\omega_{\gamma'}-2^k\omega_{\gamma}(-t)}^2\exp\sparen{2^{-k'}\sabs{2^{k'}\omega_{\gamma'}-2^k\omega_{\gamma}(-t)}^2}\\
&\leq \sum\limits_{j=1}^{\infty}2^n2^k(j+1)^{n+2}\exp\sparen{-\frac{(j-1)^2}{2}}\leq D2^k.
\end{align*}
Now it is easy to see the contribution there is only a small (or 0) contribution coming from the sets $\mathcal{C}$, $\mathcal{D}$, and $\mathcal{E}$ (since $|k-k'|\leq k_0$), and this contribution is uniformly bounded independent of $\gamma,\gamma'$. From here it follows that the second sum (\ref{piecetwoT}) is $\mathcal{O}(2^k)$. Combining the estimates on above gives 
\[
\sum\limits_{\gamma':|k'-k|\leq k_0}|b_T(\gamma,\gamma',t)|\leq D2^{k}.
\]
Since the contribution from regions 1 and 3 was uniformly bounded independent of $\gamma,\gamma'$, we find
\[
\sum\limits_{\gamma'}|b_T(\gamma,\gamma',t)|\leq D2^{k}.
\]

By symmetry we can use similar estimates to obtain the second bound in \ref{doublestar}. Again, the main difference will be that there is no need to apply Lemma \ref{propertyB} in region 2. With these estimates we conclude the theorem.  

\section{Construction of the Parametrix}
With the frame of functions established, we turn our attention to constructing an appropriate parametrix for the Cauchy problem 
\begin{align*}
&Tu(t,x)=(\partial_t^2-A(t,x,\partial_x))u(t,x)=0\\&
u(t,x)|_{t=0}=f(x)\\&
\partial_tu(t,x)|_{t=0}=h(x)
\end{align*}
where $f(x)$ and $h(x)$ are functions in $L^2(\mathbb{R}^n)$. 
We will construct operators $\mathcal{C}(t,t')$ and $\mathcal{S}(t,t')$ out of families of functions which are related to the frame functions. This section will follow the work of Hart Smith in [5] very closely. 

As earlier $\mathcal{U}(t,t')$ denotes the evolution operator associated to $\mathcal{H}^{-}=\tau-q$. Additionally  we denote  the evolution operator associated to the Hamiltonian $\mathcal{H}^{+}=\tau+q$ as $\mathcal{V}(t,t')$. Setting 
\[
\mathcal{U}(t,t')(x_{\gamma}(0),\xi_{\gamma}(0))=(x^{+}_{\gamma}(t,t'),\xi^{+}_{\gamma}(t,t'))
\]
and
\[
\mathcal{V}(t,t')(x_{\gamma}(0),\xi_{\gamma}(0))=(x^{-}_{\gamma}(t,t'),\xi^{-}_{\gamma}(t,t'))
\]
Then accordingly 
\begin{align*}
&\phi_{\gamma}^{\pm}(t,t',x)\\&=
\sparen{\frac{|\xi_{\gamma}(t)|\Delta x_{\gamma}}{2\pi}}^{\frac{n}{2}}\exp\sparen{i\xi_{\gamma}^{\pm}(t,t')\cdot(x-x^{\pm}_{\gamma}(t,t'))-|\xi_{\gamma}^{\pm}(t,t')||x-x_{\gamma}^{\pm}(t,t')|^2}
\end{align*}
and we let
\begin{align*}
\Omega_{\gamma}^{\pm}(t,t',x)=
\frac{\phi_{\gamma}^{\pm}(t,t',x)}{q(t',x_{\gamma},\xi_{\gamma})}
\end{align*} 

From these definitions we construct the following operators $\mathcal{C}(t,t')$ 
\begin{align*}
\Pi^0 \mathcal{C}(t,t')\Pi^0 f & = P^0_2 B_{\mathcal{C}}(t,t') P_1^0 f \\
=&\sum\limits_{\gamma,\gamma'}b_{\mathcal{C}}(\gamma,\gamma',t)c(\gamma')\phi_{\gamma}(x)
\end{align*}
and $\mathcal{S}(t,t')$: 
\begin{align*}
\Pi^0 \mathcal{S}(t,t')\Pi^0 f & = P^0_2 B_{\mathcal{S}}(t,t') P_1^0 f \\
=&\sum\limits_{\gamma,\gamma'}b_{\mathcal{S}}(\gamma,\gamma',t)c(\gamma')\phi_{\gamma}(x)
\end{align*}
where 
\begin{align*}
b_{\mathcal{C}}(\gamma,\gamma',t)=\frac{1}{2}\int_{\mathbb{R}^n}\overline{\phi_{\gamma}(x)}\sparen{\phi_{\gamma'}^{+}(t,t',x)+\phi_{\gamma'}^{-}(t,t',x)}\,dx
\end{align*}
and
\begin{align*}
b_{\mathcal{S}}(\gamma,\gamma',t)=\frac{1}{2}\int_{\mathbb{R}^n}\overline{\phi_{\gamma}(x)}\sparen{\Omega_{\gamma'}^{+}(t,t',x)-\Omega_{\gamma'}^{-}(t,t',x)}\,dx
\end{align*}
denote the entries of the matrices $B_{\mathcal{C}}(t,t')$ and $B_{\mathcal{S}}(t,t')$ respectively. 

\begin{thm}
$T\mathcal{C}(t,t')$ and $T\mathcal{S}(t,t')$ are bounded operators of order one and zero respectively with operator norms which are uniformly bounded on intervals where $t-t'$ is finite. 
\end{thm}

\begin{proof} By construction the matrix entries $B_{\mathcal{C}}(t,t')$ and $B_{\mathcal{S}}(t,t')$ consist of linear combinations of inner products of propagated frame functions with the frame, so the proof is an immediate extension of Theorem \ref{boundedtheorem} in section 2. 
\end{proof}
\begin{thm}
\[
\mathcal{C}(t',t')\sim I \qquad \partial_t\mathcal{C}(t',t')=0
\]
and
\[
\mathcal{S}(t',t')=0 \qquad \partial_t\mathcal{S}(t',t')\sim I
\]
to leading order. 
\end{thm}
\begin{proof}
The first statement follows directly from the definition of $\Pi^0$ and the integral calculations in section 2. For the second, we compute (where $q=q(t',x_{\gamma},\xi_{\gamma})$)
\[
((\partial_t\Omega^+(t',t',y)-\partial_t\Omega^-(t',t',y))|=\sparen{\sparen{1-\frac{q_t}{q^2}}-\sparen{-1+\frac{q_t}{q^2}}}\phi_{\gamma}(y)
\]
as on null bicharacteristics $\tau=\pm q$ by homogeneity 
\[
\frac{q_t}{q^2}=\mathcal{O}\sparen{\frac{1}{2^{k'}}}
\]
and the result follows. From the proceeding arguments $u(t,x)$ such that
\[
u(t,x)=\mathcal{S}(t,t')h(x)+\mathcal{C}(t,t')h(x) 
\]
is the desired parametrix solution to the Cauchy problem. 
\end{proof}
\begin{thm}\label{Cauchy}
If $-1\leq m\leq 2$ and $f\in H^{m+1}(\mathbb{R}^n), h\in H^{m}(\mathbb{R}^n)$ and $F\in L^1([-T,T],; H^{m}(\mathbb{R}^n))$ then there exists a $G\in L^1([-T,T],; H^{m}(\mathbb{R}^n))$ such that 
\[
u(t,x)=\mathcal{C}(t,0)f(x)+\mathcal{S}(t,0)h(x)+\int\limits_0^t(\mathcal{S}(t,s)G(s,x)\,ds
\]
and 
\[
\norm{G}_{L^1([-T,T],; H^{m}(\mathbb{R}^n))}\leq C(T)\sparen{\norm{f}_{H^{m+1}(\mathbb{R}^n)}+\norm{h}_{H^{m}(\mathbb{R}^n)}+\norm{F}_{L^1([-T,T],; H^{m}(\mathbb{R}^n))}}
\]
solves the Cauchy problem
\begin{align*}
&Tu(t,x)=(\partial_t^2-A(t,x,\partial_x))u(t,x)=F(t,x)\\&
u(t,x)|_{t=0}=f(x)\\&
\partial_tu(t,x)|_{t=0}=h(x)
\end{align*}
in the weak sense. If $f$ and $h$ are both identically zero and $F$ is also zero for $\forall t\in [-T,T]$ then $G$ and $u$ will vanish as well. 
\end{thm}
\begin{proof}
As per Smith, we will show the existence of such a $G$ using Volterra iteration. Assuming $G\in L^1([-T,T],; H^{m}(\mathbb{R}^n))$, we let 
\[
v(t,x)=\int\limits_0^t\mathcal{S}(t,s)G(s,x)\,ds
\]
Because $\mathcal{S}(t,t')$ and $\partial_t\mathcal{S}(t,t')$ are both strongly continuous operators and $\mathcal{S}(t,t)=0$ we have $v(t,x)$ is in $C([-T,T];H^{m+1}(\mathbb{R}^n))\cap C^1([-T,T];H^{m}(\mathbb{R}^n))$ 
and also 
\[
\partial_tv(t,x)=\int\limits_0^t\partial_t\mathcal{S}(t,s)G(s,x)\,ds
\]
so it follows
\[
v(0,x)=0 \qquad \partial_tv(t,x)|_{t=0}=0
\]
Furthermore differentiating in the sense of distributions we obtain
\[
\partial_t^2v(t,x)=G(t,x)+\int\limits_0^t\partial_t^2\mathcal{S}(t,s)G(s,x)\,ds
\]
We can conclude $u(t,x)$ of the form in Theorem \ref{Cauchy} is a weak solution to the Cauchy problem if the following Volterra equation 
\begin{equation}\label{Volterra}
G(t,x)+\int\limits_0^{t}TS(t,s)G(s,x)\,ds=F(t,x)-T\sparen{\mathcal{S}(t,0)f(x)+\mathcal{S}(t,0)h(x)}
\end{equation}
holds.  Equation (\ref{Volterra}) can be solved by iteration since the operator norm of $S(t,s)$ is uniformly bounded on finite intervals of time by Theorem \ref{boundedtheorem}. Setting 
\begin{equation}\label{eqseries}
G(t,x)=F(t,x)+\sum\limits_{n=1}^{\infty}G_n(t,x)
\end{equation}
with 
\begin{align*}
G_n(t,x)=\int\limits_{0}^{t}\int\limits_{0}^{s_1} . . .  \int\limits_{0}^{s_{n-1}}S(t,s_1)S(s_1,s_2) . . . \\&
S(s_{n-1},s_n)F(s_n,x)\,ds_n. . . \,ds_1
\end{align*}
then $G(t,x)$ is a solution to the equation
\[
G(t,x)+\int\limits_{0}^{t}S(t,s)G(s,x)\,ds=F(t,x)
\]
As the series in (\ref{eqseries}) converges in  $L^1([-T,T],; H^{m}(\mathbb{R}^n))$ with norm bounded by $\exp(TC(T)\norm{F})$, this finishes the Theorem. 
\end{proof}

\section*{Appendix 1}

It is well known that 
\begin{align}\label{absbd}
\int\limits_{\mathbb{R}^n}\exp\sparen{iy\cdot \eta}\exp\sparen{-c y^2}\,dy =\sparen{\frac{\pi}{c}}^{\frac{n}{2}}\exp\sparen{-\frac{\eta^2}{4c}} 
\end{align}
We will use this fact to help us evaluate Integrals of the form 
\begin{align}\label{Dabsbd}
\int\limits_{\mathbb{R}^n}(y+b)\exp\sparen{iy\cdot \eta}\exp\sparen{-c y^2}\,dy
\end{align}
and
\begin{align}\label{Tabsbd}
\int\limits_{\mathbb{R}^n}|y+b|^2\exp\sparen{iy\cdot \eta}\exp\sparen{-c y^2}\,dy
\end{align}
Recall that for $c$ a constant, $\eta, y\in \mathbb{R}^n$ and $\alpha$ a multi-index with $n$ components 
\begin{align}\label{byparts}
i^{|\alpha|}\eta^{\alpha}\int\limits_{\mathbb{R}^n}\exp\sparen{-c y^2}\exp\sparen{i\eta\cdot y}\,dy=(-1)^{|\alpha|}\int\limits_{\mathbb{R}^n}\partial_y^{\alpha}(\exp\sparen{-cy^2})\exp\sparen{i\eta\cdot y}\,dy 
\end{align}
and also
\begin{align*}
&\frac{\partial}{\partial y}\exp\sparen{-c y^2}=-2c y \exp\sparen{-c y^2}\\ &
\frac{\partial^2}{\partial y^2}\exp\sparen{-c y^2}=\sparen{-2c+4c^2y^2} \exp\sparen{-c y^2}
\end{align*}
With these equalities in mind, (\ref{Dabsbd}) is equal
\begin{align*}
&\frac{1}{c}\int\limits_{\mathbb{R}^n}\partial_y(\exp\sparen{-cy^2})\exp\sparen{i\eta\cdot y}\,dy+b\int\limits_{\mathbb{R}^n}(\exp\sparen{-cy^2})\exp\sparen{i\eta\cdot y}\,dy \\&
=\sparen{\frac{\pi}{c}}^{\frac{n}{2}}\exp\sparen{-\frac{\eta^2}{4c}}\sparen{\frac{i\eta}{2c}+b}
\end{align*}
We can also expand and re-write the integral in (\ref{Tabsbd}) so it is equal
\begin{align*}
&\frac{1}{4c^2}\int\limits_{\mathbb{R}^n}\partial_y^{2}(\exp\sparen{-cy^2})\exp\sparen{i\eta\cdot y}\,dy +\frac{b}{c}\int\limits_{\mathbb{R}^n}\partial_y(\exp\sparen{-cy^2})\exp\sparen{i\eta\cdot y}\,dy \\&
+\sparen{b^2+\frac{1}{2c}}\int\limits_{\mathbb{R}^n}(\exp\sparen{-cy^2})\exp\sparen{i\eta\cdot y}\,dy
\end{align*}
Using the integration by parts formula \ref{byparts} is just
\begin{align}
\sparen{\frac{\pi}{c}}^{\frac{n}{2}}\exp\sparen{-\frac{\eta^2}{4c}}\sparen{-\frac{\eta^2}{4c^2}+\frac{ib\eta}{c}+b^2+\frac{1}{2c}}
\end{align}

\section{Appendix 2}
Given an integer valued function $h(\alpha)$, $h(\alpha)$ may be estimated by the Euler summation formula
\begin{equation}\label{Euler}
\sum\limits_{a\leq \alpha\leq b}h(\alpha)=\int\limits_{a}^{b}h(x)\,dx +\sum\limits_{j=1}^{m}\frac{B_j}{j!}h^{j-1}(x)|_{x=a}^{x=b} +R_m
\end{equation}
where $B_j$ is the $j^{th}$ Bernoulli number and $h^{j}(x)$ denotes the $j^{th}$ derivative of $h(x)$.  The remainder $R_m$ is defined as 
\[
R_{m}=(-1)^{m+1}\int\limits_{\mathbb{R}}\frac{B_m(\{x\})}{m!}h^m(x)\,dx
\]
The notation $\{x\}$ denotes the fractional part of $x$, and $B_m(\{x\})$ denotes the $m^{th}$ Bernoulli polynomial. Formula (\ref{Euler}) is derived in \emph{Concrete Mathematics}, cf Ref[1].

Fix $\epsilon_0\in \mathbb{R}^n$ and $\lambda \in \mathbb{R},$ with $\lambda\geq 1$. We wish to use the Euler summation formula to estimate the sums:
\begin{align} \label{thetazero}
\sum\limits_{\alpha\in\mathbb{Z}^n}\exp\sparen{-\frac{|\alpha-\epsilon_0|^2}{\lambda}} 
\end{align}
and
\begin{align}\label{thetatwo}
\sum\limits_{\alpha\in\mathbb{Z}^n}\frac{|\alpha-\epsilon_0|^2}{\lambda}\exp\sparen{-\frac{|\alpha-\epsilon_0|^2}{\lambda}}
\end{align}
in terms of the parameter $\lambda$. Since the variables $\alpha_1,\alpha_2, . . .\alpha_n$ are indexed independent of each other, we may re-write (\ref{thetazero}) as 
\begin{equation}\label{ind}
\sum\limits_{\alpha_i\in \mathbb{Z}}\sparen{\prod\limits_{i=1}^n\exp\sparen{-\frac{|\alpha_i-\epsilon_{0_i}|^2}{\lambda}}}=\prod\limits_{i=1}^n\sparen{\sum\limits_{\alpha_i\in \mathbb{Z}}\exp\sparen{-\frac{|\alpha_i-\epsilon_{0_i}|^2}{\lambda}}}
\end{equation}
We apply the Euler summation formula with $m=2$ to the sum in parenthesis on the right hand side of (\ref{ind}) so that $h(x)=\exp\sparen{-\frac{|x-\epsilon_{0_i}|^2}{\lambda}}$.
Let $g(x)=\exp\sparen{-x^2}$ then via change of variables $x=\sqrt{\lambda}(u+\epsilon_{0_i})$
\[
R_2=\frac{-1}{2}\int\limits_{\mathbb{R}}B_2(\{x\})h''(x)\,dx=\frac{-1}{2\sqrt{\lambda}}\int\limits_{\mathbb{R}}B_2(\{\sqrt{\lambda}(u+\epsilon_{0_i})\})g''(u)\,du
\]
By properties of the Bernoulli numbers (again, cf Ref[1]): 
\[
|B_2(\{\sqrt{\lambda}(u+\epsilon_{0_i})\})|\leq B_2= \frac{1}{6}
\]
Integrating by parts gives 
\[
\sabs{\int\limits_{\mathbb{R}}g''(u)\,du} \leq \int\limits_{\mathbb{R}}(4u^2+2)e^{-u^2}\,du=6\sqrt{\pi}
\]
\[
|R_2|\leq \sqrt{\frac{\pi}{2\lambda}}
\]
The second term on the right hand side in the Euler summation formula vanishes: 
\[
\sum\limits_{j=1}^{m}\frac{B_j}{j!}h^{j-1}(x)|_{x=-\infty}^{x=\infty}=0
\]
As a result
\[
\sum\limits_{\alpha\in\mathbb{Z}^{n}}\exp\sparen{-\frac{|\alpha-\epsilon_{0}|^2}{\lambda}}\leq \sparen{2\pi\lambda}^{\frac{n}{2}} 
\]
The second sum (\ref{thetatwo}) can be re-written as
\[
\sum\limits_{i=1}^{n}\sparen{\sum\limits_{\alpha_i\in \mathbb{Z}}\frac{|\alpha_i-\epsilon_{0_i}|^2}{\lambda}\exp\sparen{-\frac{|\alpha_i-\epsilon_{0_i}|^2}{\lambda}}}\sum\limits_{\alpha'\in\mathbb{Z}^{n-1}}\exp\sparen{-\frac{|\alpha'-\epsilon'_{0}|^2}{\lambda}}
\]
 Here $\alpha'=(\alpha_1,\alpha_2,. . .  \widehat{\alpha}_i, . . . \alpha_{n})$ and $\epsilon_0'=(\epsilon_{0_1},\epsilon_{0_2},. . .  \widehat{\epsilon}_{0_i}, . . . \epsilon_{0_n})$
Applying the Euler summation formula to
\begin{equation}\label{sumthetatwo}
\sum\limits_{\alpha_i\in \mathbb{Z}}\frac{|\alpha_i-\epsilon_{0_i}|^2}{\lambda}\exp\sparen{-\frac{|\alpha_i-\epsilon_{0_i}|^2}{\lambda}}
\end{equation}
gives  (\ref{sumthetatwo}) is also $\mathcal{O}(\sqrt{\lambda})$. This follows since 
\begin{equation}
 \int\limits_{-\infty}^{\infty}\frac{x^2}{\lambda}\exp\sparen{-\frac{x^2}{\lambda}}\,dx=\sqrt{2\pi \lambda} 
\end{equation}
The details are left to the reader. Therefore (\ref{thetatwo}) is $\mathcal{O}((\lambda)^{\frac{n}{2}})$ as well.

\end{document}